\documentclass[12pt]{amsart}
\usepackage[margin=1.5in]{geometry}
\usepackage{tikz,tikz-cd}
\usepackage{url}
\usepackage{graphicx}
\usepackage{graphicx}
\usepackage{color}
\usepackage{times}
\usepackage{cite}
\usepackage{enumerate,latexsym}
\usepackage{latexsym}
\usepackage{amsmath,amssymb}
\usepackage{graphicx}
\usepackage{lipsum}
\usepackage{amsthm}
\usepackage{verbatim}
\newtheorem{thm}{Theorem}[section]
\newtheorem*{theorem*}{Theorem}
\newtheorem*{acknowledgement*}{Acknowledgement}
\newtheorem{cor}[thm]{Corollary}

\newtheorem{lem}[thm]{Lemma}
\newtheorem{prop}[thm]{Proposition}
\theoremstyle{definition}

\theoremstyle{remark}
\newtheorem{rem}[thm]{Remark}

\numberwithin{equation}{section}
\theoremstyle{definition}
\newtheorem{definition}{Definition}[section]
\theoremstyle{remark}
\newtheorem{remark}{Remark}[section]

\newcommand{\set}[1]{\left\{#1\right\}}
\newcommand{\Real}{\mathbb R}

\newcommand{\dist}[0]{\mathrm{dist}}

\title[Minimal Surfaces and Entropy in $\mathbb{CH}^{n+1}$]{Minimal surfaces and Colding-Minicozzi entropy in complex hyperbolic space }
\author{Jacob Bernstein}
\address{Department of Mathematics, Johns Hopkins University, 3400 N. Charles Street, Baltimore, MD 21218}
\email{bernstein@math.jhu.edu}

\author{Arunima Bhattacharya}
\address{Department of Mathematics, the University of North Carolina at Chapel Hill, Phillips Hall, Chapel Hill, NC 27514 }
\email{arunimab@unc.edu}

\begin{document}

\begin{abstract}
We study notions of asymptotic regularity for a class of minimal submanifolds of complex hyperbolic space that includes minimal Lagrangian submanifolds.  
As an application, we show a relationship between an appropriate formulation of Colding-Minicozzi entropy and a quantity we call the $CR$-volume that is computed from the asymptotic geometry of such submanifolds. 
\end{abstract}
\maketitle

 \section{Introduction}

 In \cite{BernsteinBhattacharya}, the authors generalized the Colding-Minicozzi entropy \cite{Coldinga} to submanifolds of Cartan-Hadamard manifolds. In this article, we study the specific case where the ambient manifold is complex hyperbolic space, $\mathbb{CH}^{n+1}$.  Recall, this is the $(2n+2)$-dimensional complete simply-connected K\"{a}hler-Einstein manifold with metric $g_{\mathbb{CH}}$ which has constant negative holomorphic curvature;  we adopt the convention that the sectional curvatures lie in $[-4,-1]$.   This leads to a relationship between the entropy of certain minimal submanifolds (e.g., minimal Lagrangian submanifolds) and a quantity we call the $CR$-volume which is associated to their asymptotic geometry.  This quantity is an analog, in the context of $CR$-geometry,  of the conformal volume of Li and Yau \cite{LiYau}; it is also related to the visual volume of Gromov \cite{GromovFilling} -- see Definition \ref{CRVol} in Section \ref{CRVolSec} and Section \ref{EntropySec}.

Our study requires appropriate compactifications of $\mathbb{CH}^{n+1}$; we describe two natural choices in Section \ref{CompactSec}.  Both lead to well-defined and equivalent notions of \emph{ideal boundary}, $\partial_\infty \mathbb{CH}^{n+1}$.  Importantly, $\partial_\infty \mathbb{CH}^{n+1}$ comes equipped with a natural $CR$-structure modeled on $\mathbb{S}^{2n+1}$ viewed as the boundary of the unit complex ball $\mathbb{B}_{\mathbb{C}}^{n+1}$, and in fact an equivalence class of Sasaki structures. 
Let $\Sigma\subset \mathbb{CH}^{n+1}$ be a smooth proper minimal submanifold. Such a $\Sigma$ is necessarily non-compact, and therefore a natural assumption is that its asymptotic geometry can be modeled, via the compactifications of $\mathbb{CH}^{n+1}$, on a submanifold, $\partial_\infty \Sigma\subset \partial_\infty \mathbb{CH}^{n+1}$ -- see Section \ref{MinCompactSec} for a detailed discussion.   One may then impose additional conditions on the asymptotic regularity and geometry of $\Sigma$. For instance, $\Sigma$ is \emph{asymptotically horizontal} if $\partial_\infty\Sigma$ is a horizontal submanifold of $\partial_\infty\mathbb{CH}^{n+1}$ and \emph{asymptotically Legendrian} when $\partial_\infty\Sigma$ is Legendrian -- i.e.,  horizontal and of maximal dimension. Indeed, a Lagrangian submanifold of sufficient asymptotic regularity is asymptotically Legendrian -- see Lemma \ref{IsotropIsHorizLem}. 
We refer to Definitions \ref{def_asymreg} and  \ref{def_weakasym} in  Section \ref{CompactSec_1} for specifics.


Our first result is to show that relatively weak assumptions on the asymptotic regularity of certain minimal submanifolds imply stronger asymptotic regularity.
\begin{thm}\label{main1}
	Suppose that $\Sigma\subset \mathbb{CH}^{n+1}$ is an $m$-dimensional  minimal submanifold.  If $\Sigma$ is weakly $C^2$-asymptotically regular and weakly asymptotically horizontal, then  it is $C^1$-asymptotically regular and strongly horizontal.  
\end{thm}

Using this, we obtain a relationship between the generalized Colding-Minicozzi entropy of these submanifolds and the $CR$-volume of their asymptotic boundaries.  

\begin{definition}[$\kappa$-entropy \cite{BernsteinBhattacharya}]
     Suppose that $\Sigma$ is an $m$-dimensional submanifold of the $(m+k)$-dimensional Cartan-Hadamard manifold $(M,g)$. Let
$$
\Phi_{m, \kappa}^{t_0,x_0}(t,x)=K_{m,\kappa}(t_0-t, \dist_{g}(x,x_0))
$$ where $K_{m,\kappa}$ are functions defined in \cite[pg 2]{BernsteinBhattacharya} following \cite{Coldinga, BernsteinHypEntropy} and is related to the heat kernel of $\mathbb{H}^m$.
The Colding-Minicozzi $\kappa$-entropy of $\Sigma$ in $(M,g)$ is given by
$$
\lambda_g^\kappa[\Sigma]=\sup_{x_0\in M, \tau>0}\int_{\Sigma} \Phi_{m, \kappa}^{0, x_0} (-\tau, x) dVol_{\Sigma}(x)=\sup_{x_0\in M, \tau>0}\int_{\Sigma} \Phi_{m, \kappa}^{\tau, x_0} (0, x) dVol_{\Sigma}(x).
$$
\end{definition}
When $\kappa=0$ and $(M,g)=(\Real^{m+k}, g_{\Real})$ is Euclidean space, this is the usual Colding-Minicozzi entropy, $\lambda[\Sigma]$, of $\Sigma$. When $\kappa=1$ and $(M,g)=(\mathbb{H}^{m+k}, g_{\mathbb{H}})$ is hyperbolic space, this is the entropy in hyperbolic space, $\lambda_{\mathbb{H}}[\Sigma]$, introduced in \cite{BernsteinHypEntropy}.   By \cite{BernsteinBhattacharya},  $\lambda_g^\kappa$ is monotone non-increasing along reasonable mean curvature flows in a Cartan-Hadamard manifold with sectional curvatures bounded above by $-\kappa^2$.  As $\mathbb{CH}^{n+1}$ has sectional curvatures in $[-4,-1]$, we make the following definition.

\begin{definition}[Complex hyperbolic entropy]
    The Colding-Minicozzi entropy of an $m$-dimensional submanifold $\Sigma\subset \mathbb{CH}^{n+1}$ is given by
$$
\lambda_{\mathbb{CH}}[\Sigma]=\lambda_{g_{\mathbb{CH}}}^1[\Sigma]=\sup_{x_0\in \mathbb{CH}^{n+1}, \tau>0} \int_{\Sigma} \Phi_{m, 1}^{\tau, x_0} (0, x) dV_\Sigma(x)
$$
where $\lambda_{g_{\mathbb{CH}}}^1$ is the $\kappa$-entropy on $\mathbb{CH}^{n+1}$ corresponding to an upper bound on sectional curvatures of $-1$.
\end{definition}

\begin{thm}\label{main2}
If $\Sigma \subset \mathbb{CH}^{n+1}$ is an $m$-dimensional submanifold that is weakly $C^1$-asymptotically regular and weakly asymptotically horizontal, then
$$
|\mathbb{S}^{m-1}|_{\Real} \lambda_{\mathbb{CH}}[\Sigma]\geq \lambda_{CR}[\partial_\infty \Sigma].
$$
If $\Sigma$ is also weakly $C^2$-asymptotically regular and minimal, then equality holds.
\end{thm}

\begin{remark}
In \cite[Theorem 1.5]{BernsteinHypEntropy}, an analogous result for appropriate submanifolds of hyperbolic space was obtained: If $\Sigma$ is an $m$-dimensional submanifold of $\mathbb{H}^{m+k}$ that is regular up to the ideal boundary, then there is an inequality relating, $\lambda_{c}[\partial_\infty\Sigma]$, the conformal volume of the ideal boundary of $\Sigma$ and $\lambda_{\mathbb{H}}[\Sigma]$, the Colding-Minicozzi entropy of $\Sigma$ in hyperbolic space.  This was applied in \cite{JYaoRelArea, JYaoMP} to show a type of topological uniqueness for certain minimal hypersurfaces in $\mathbb{H}^{n+1}$. 
\end{remark}

%
We conclude by discussing the applicability of Theorem \ref{main2}.  The most basic examples are the totally geodesic Lagrangian submanifolds, $\Sigma=\Sigma_{p,L}$, that, for any $p\in \mathbb{CH}^{n+1}$ and Lagrangian subspace $L \subset T_p \mathbb{CH}^{n+1}$,  are uniquely determined by $L=T_p \Sigma$.  These $\Sigma$ are  $C^\infty$-asymptotically regular, strongly horizontal with $\partial_\infty \Sigma$ corresponding, modulo a $CR$-automorphism,  to a totally geodesic Legendrian sphere $\mathbb{S}^n\subset \mathbb{S}^{2n+1}$.  They also satisfy $\lambda_{\mathbb{CH}}[\Sigma]=1$. 

The fact that each $\Sigma_{p,L}$ is intrinsically, $\mathbb{H}^{n+1}$, the hyperbolic space with curvature $-1$, yields more examples.  Indeed,  work of Anderson \cite{andersonCompleteMinimalVarieties1982,andersonCompleteMinimalHypersurfaces1983a}, Hardt-Lin \cite{hardtRegularityInfinityAbsolutely1987}, Lin \cite{linDirichletProblemMinimal1989}, and Tonegawa \cite{tonegawaExistenceRegularityConstant1996}, give many minimal hypersurfaces in $\mathbb{H}^{n+1}$ with good asymptotic regularity.  This is done by solving an asymptotic Plateau problem for any $C^2$-regular hypersurface in $\partial_\infty \mathbb{H}^{n+1}$; note that, as observed in \cite{tonegawaExistenceRegularityConstant1996}, higher asymptotic regularity is a subtle issue. See the survey of Coskunuzer \cite{coskunuzerAsymptoticPlateauProblem2013} for a thorough overview.  Embedding these into $\Sigma_{p,L}$ yields weakly $C^2$-asymptotically regular and asymptotically horizontal $n$-dimensional minimal submanifolds in $\mathbb{CH}^{n+1}$.

In another direction, in \cite[Theorem 1]{castroMinimalLagrangianSubmanifolds2002}, the authors show the existence of a family of rotationally symmetric minimal Lagrangian submanifolds of $\mathbb{CH}^{n+1}$ that are topologically  $\mathbb{R}\times \mathbb{S}^n$.  One verifies that, when $n$ is even, they are $C^\infty$-asymptotically regular while, when $n$ is odd, they are only $C^{2}$-asymptotically regular.  In all dimensions, these submanifolds are weakly $C^\infty$-asymptotically regular and strongly horizontal.  Moreover, their ideal boundary corresponds to a pair of totally geodesic Legendrian spheres $\mathbb{S}^n$ contained in $\mathbb{S}^{2n+1}$.   There seems to be few other constructions of smooth minimal Lagrangian submanifolds in $\mathbb{CH}^{n+1}$ -- see however \cite{loftinMinimalLagrangianSurfaces2013}.  If one allows (interior) singularities then, by taking a cone, any minimal Legendrian in $\mathbb{S}^{2n+1}$ gives rise to a (singular) minimal Lagrangian submanifold of $\mathbb{CH}^{n+1}$ -- see Corollary \ref{MinimalConeLem}.  There are many examples of such $\Gamma$ in $\mathbb{S}^{2n+1}$ -- for instance,  the so-called Clifford tori \cite[Ex. 3.16]{Haskins} -- see also \cite{HaskinsKapouleas2}.  In $\mathbb{S}^5$ there is a particularly rich collection of examples, including those of higher genus -- see \cite{CarberryMcIntosh, McIntosh, HaskinsKapouleas}.  Note that Theorem \ref{main2} holds for minimal submanifolds with interior singularities.

Finally, for any submanifold $\Gamma\subset \partial_\infty \mathbb{CH}^{n+1}$, and in particular any horizontal submanifold, it follows from \cite[Theorem 4.3]{bangertTrappingQuasiminimizingSubmanifolds1996} that there is a (singular) area-minimizing current asymptotic to $\Gamma$ in a certain, weak, sense.  To the authors' knowledge,  additional asymptotic regularity has not been established in $\mathbb{CH}^{n+1}$ for the solutions of \cite{bangertTrappingQuasiminimizingSubmanifolds1996}, especially for those of high codimension -- cf.  \cite{linAsymptoticBehaviorAreaminimizing1989}, which shows asymptotic regularity for high-codimension minimizers in $\mathbb{H}^{n+1}$.  It would be interesting to produce solutions with sufficient asymptotic regularity in this fashion.

\subsection*{Acknowledgments}
The first author was partially supported by the NSF Grant DMS-1904674 and  DMS-2203132 and the Institute for Advanced Study with funding provided by the Charles Simonyi Endowment.
The second author acknowledges the support of the AMS-Simons Travel Grant and funding provided by the Bill Guthridge Distinguished Professorship Fund. 

\section{Geometric background}\label{sec_basics}


We recall some geometric backgrounds and carry out some basic computations.  In particular, we recall the natural contact, Sasaki, and CR structures that come from viewing $\mathbb{S}^{2n+1}\subset \Real^{2n+2}\simeq \mathbb{C}^{n+1}$  as the boundary of the unit ball $\mathbb{B}^{2n+2}\simeq \mathbb{B}_{\mathbb{C}}^{n+1}$.  This is important as these structures naturally arise on the ideal boundary of complex hyperbolic space.  Our main source is \cite{BlairBookNew, Tanno1989}. 

\subsection{Basic constructs on Euclidean space}

Let us first introduce notation for various structures on  $\Real^{2n+2}\simeq \Real^{n+1}_{\mathbf{x}}\times \Real^{n+1}_{\mathbf{y}}\simeq \mathbb{C}^{n+1}_{\mathbf{z}}$.  Here we make the identification for $(\mathbf{x}, \mathbf{y})\in  \Real^{n+1}\times  \Real^{n+1}$ with $\mathbb{C}^{n+1}$ via
$$
\mathbf{z}=\mathbf{x}+i \mathbf{y}.
$$
In particular, the Euclidean coordinates given by $x_1, y_1, \ldots, x_{n+1},  y_{n+1}$, correspond to the holomorphic coordinates $z_1=x_1+i y_1, \ldots, z_{n+1}=x_{n+1}+iy_{n+1}$.

Denote the usual Euclidean Riemannian metric and symplectic form by
\begin{align*} g_{\Real}=\sum_{j=1}^{n+1} \left( dx_j^2+dy_j^2\right)\mbox{ and } \omega_{\Real}=\sum_{j=1}^{n+1} dx_j\wedge dy_j.
 \end{align*}
Let $J_{\Real}$ be the associated almost complex structure on $\Real^{2n+2}$  defined by
$$
\omega_{\Real}(X,Y)=g_{\Real}(X,J_{\Real}(Y))
$$
so
$$
J_{\Real}\left(\frac{\partial}{\partial x_j }\right)=-\frac{\partial}{\partial y_j} \mbox{ and } J_{\Real}\left(\frac{\partial}{\partial y_j}\right)=\frac{\partial}{\partial x_j}.
$$
In this convention the complexification of $J_{\Real}$ satisfies $J_{\Real}\left(\frac{\partial}{\partial z_j}\right)=-i \frac{\partial}{\partial z_j}$.
Likewise, 
$$
dx_j\circ J_{\Real}=dy_j \mbox{ and } dy_j \circ J_{\Real} =-dx_j.
$$

Let $r:\Real^{2n+2}\to \Real$ be the radial function defined by
$$
r^2=x_1^2+y_1^2+\cdots+x_{n+1}^2+y_{n+1}^2=|\mathbf{x}|^2+|\mathbf{y}|^2=|\mathbf{z}|^2.
$$
One readily computes that  
$$
dr=\frac{1}{r} \left(x_1 dx_1 +y_1 dy_1+\cdots +x_{n+1} dx_{n+1}+y_{n+1} dy_{n+1}\right)
$$
is a smooth one-form on $\Real^{2n+2}\setminus \set{0}$. 
Set
$$
\theta= \frac{1}{r} dr \circ J_{\mathbb{R}}=\frac{1}{r^2} \left(x_1 dy_1 -y_1 dx_1+\cdots +x_{n+1} dy_{n+1}-y_{n+1} dy_{n+1}\right).
$$
Define a symmetric $(0,2)$ tensor field on $\Real^{2n+2}\setminus \set{0}$ by, 
\begin{align*}
\eta&=\frac{1}{r^2}\sum_{j, k=1, j\neq k}^{n+1} \left( (x_j x_k +y_j y_k) (dx_j dx_k+dy_j dy_k) +2 x_j y_k (dx_j dy_k -dy_j dx_k)\right).
	\end{align*}
One readily computes that
$$
g_{\Real}=dr^2+r^2\left( \theta^2+\eta\right) \mbox{ and } \omega_{\Real}= \frac{1}{2}r^2d\theta + r dr \wedge \theta.
$$
Hence, for $X,Y\in T_p \Real^{2n+2}$, 
\begin{equation} \label{LeviFormEqn}
 -\frac{1}{2} d\theta(X, J_{\Real}(Y))= -\frac{1}{r^2}\omega_{\Real}(X, J_{\Real}(Y))+ \frac{1}{r} (dr \wedge \theta)(X, J_{\Real}(Y))= \eta(X,Y).
\end{equation}
It is convenient to introduce vector fields
$$
\mathbf{X} =\sum_{j=1}^{n+1} \left(x_j  \frac{\partial}{\partial x_j}+y_j  \frac{\partial}{\partial y_j}\right) \mbox{ and }\mathbf{T}= \sum_{j=1}^{n+1} \left(x_j  \frac{\partial}{\partial y_j}-y_j  \frac{\partial}{\partial x_j}\right) =-J_{\Real}(\mathbf{X})
$$
where $\mathbf{X}$ is the position vector and $\mathbf{T}$ is a Killing vector field. 
They satisfy
$$
dr(X)= \frac{1}{r}g_{\Real}(X, \mathbf{X}) \mbox{ and } \theta(X)=\frac{1}{r^2} g_{\Real}(X, \mathbf{T}).
$$

One readily checks that if $X$ is any vector field on $\Real^{2n+2}$ and $\nabla^{\Real}$ is the Levi-Civita connection of $g_{\Real}$, then, because $J_{\Real}$ is $\nabla^\Real$ parallel, one has
$$
\nabla_X^{\Real} \mathbf{X}= X \mbox{ and }\nabla_{X}^\Real \mathbf{T}=\nabla_{X}^\Real \left(-J_{\Real}(\mathbf{X})\right)=-J_{\Real}(\nabla_X^\Real \mathbf{X})=-J_{\Real}(X).
$$
\subsection{Contact, Sasaki and CR geometry of $\mathbb{S}^{2n+1}$}
By thinking  of $\mathbb{S}^{2n+1}$ as the boundary of the ball $\mathbb{B}^{2n+2}$ we may endow it with a natural Sasaki structure.  This comes along with associated CR and contact structures.  In order to understand the boundary behavior of complex hyperbolic space it will be helpful to understand the interaction of these structures as well as their symmetries.
 
To that end, let $\hat{\theta}$ be the pullback of $\theta$ to $\mathbb{S}^{2n+1}$. This is readily seen to be a contact form on $\mathbb{S}^{2n+1}$ with Reeb vector field $\hat{\mathbf{T}}$, the restriction of the tangential vector field $\mathbf{T}$.  Denote by $\mathcal{H}\subset T\mathbb{S}^{2n+1}$ the contact distribution associated to $\hat{\theta}$, that is, the vector bundle over $\mathbb{S}^{2n+1}$ satisfying,  for each $p\in \mathbb{S}^{2n+1}$, 
\begin{equation}
\mathcal{H}_p=\ker \hat{\theta}_p= \set{X\in T_p \mathbb{S}^{2n+1}: \hat{\theta}_p(X)=0}\subset T_p \mathbb{S}^{2n+1}. \label{Hp}
\end{equation}

Let $g_{\mathbb{S}}$ denote the round metric on $\mathbb{S}^{2n+1}$ induced from $g_{\Real}$.  It is clear that $\mathcal{H}$ is $g_{\mathbb{S}}$ orthogonal to $\hat{\mathbf{T}}$.  It follows that $J_{\mathcal{H}}=J_{\Real}|_\mathcal{H}$ is a bundle automorphism of $\mathcal{H}$.  We may extend this to a bundle map $J_{\mathbb{S}}: T\mathbb{S}^{2n+1}\to T\mathbb{S}^{2n+1}$ by setting $J_{\mathbb{S}}(\hat{\mathbf{T}})=\mathbf{0}$.  

The triple $(J_{\mathbb{S}}, \hat{\mathbf{T}}, \hat{\theta})$ is, in the language of \cite{BlairBookNew}, an \emph{almost contact structure} on $\mathbb{S}^{2n+1}$.  In fact, together with the metric $g_{\mathbb{S}}$ this is an \emph{almost contact metric structure} and this almost contact metric structure is also \emph{Sasakian}. Indeed, by \cite[Theorem 6.3]{BlairBookNew}, if  $\nabla^{\mathbb{S}}$ is the Levi-Civita connection of $g_{\mathbb{S}}$, then it suffices to check
$$
(\nabla^{\mathbb{S}}_X  J_{\mathbb{S}})  Y= g_{\mathbb{S}}(X,Y) \hat{\mathbf{T}} -\hat{\theta}(Y) X,
$$
 holds for $X,Y\in T_p \mathbb{S}^{n+1}$.
This follows from $\nabla^\Real_X J_{\Real}=0$. Hence, for $X\in T_p \mathbb{S}^{2n+1}$, 
$$
\nabla_X^{\mathbb{S}} \hat{\mathbf{T}}= -J_{\mathbb{S}}(X).
$$

In a similar vein, by using the identification $\mathbb{S}^{2n+1}\simeq \partial \mathbb{B}_{\mathbb{C}}^{n+1}$ where 
$$
\mathbb{B}^{n+1}_{\mathbb{C}}=\set{(z_1, \ldots, z_{n+1}): |z_1|^2+\ldots +|z_{n+1}|^2<1}\subset \mathbb{C}^{n+1},
$$
is the unit complex ball, 
one may interpret $(\mathcal{H}, J_{\mathbb{S}})$ as a CR-structure.  In this case, $\hat{\theta}$ is a  pseudo-convex pseudo-hermitian form, as the \emph{Levi form}, $L_{\hat{\theta}}$, is a positive definite inner product on $\mathcal{H}$ -- see \cite{dragomirDifferentialGeometryAnalysis2006}.  Indeed, for $X,Y\in \mathcal{H}_p$, it is given\footnote{  we use the convention that $
	d\beta(X,Y)=  X(\beta(Y))-Y(\beta(X))-\beta([X,Y])$, though $
	d\beta(X,Y)= \frac{1}{2} \left( X(\beta(Y))-Y(\beta(X))-\beta([X,Y])\right)
	$ is also common in the literature}
	by 
$$
L_{\hat{\theta}}(X,Y)=-\frac{1}{2}(d\hat{\theta})(X, J_{\mathbb{S}}(Y))=-\frac{1}{2}d\theta(X,J_{\Real}(Y))={\eta}(X,Y)
$$
where we used that $X$ and $Y$ are in the kernel of $dr$ and $\theta$.  Hence, $L_{\hat{\theta}}=\hat{\eta}$ where $\hat{\eta}$ is the pullback of $\eta$ to $\mathbb{S}^{2n+1}$. Moreover, the \emph{Webster metric} associated to this data on $\mathbb{S}^{2n+1}$ recovers the standard metric, i.e., 
$$
\hat{\theta}^2+L_{\hat{\theta}}=g_{\mathbb{S}}.
$$

Given a $C^1$ function defined on $\mathbb{S}^{2n+1}$ we denote
$$
\nabla^{\mathcal{H}} f= \nabla^{\mathbb{S}} f - g_{\mathbb{S}} (\nabla^{\mathbb{S}}f, \hat{\mathbf{T}})\hat{\mathbf{T}}
$$
where $\mathcal{H}$ is the tangential component of the gradient.

\subsection{Complex automorphisms of unit ball in $\mathbb{C}^{n+1}$ and CR-automorphisms of $\mathbb{S}^{2n+1}$}
We denote by $Aut_{\mathbb{C}}(\mathbb{B}_{\mathbb{C}}^{n+1})$ the set of biholomorphic automorphisms of the unit disk.  We refer to \cite{RudinDiskBook} and \cite{GoldmanCHBook} for properties of these automorphisms, but summarize some of the needed facts.

First of all, we observe that for any $A\in \mathbf{U}(n+1)$, a unitary matrix, the map
$$
\Phi_{A}: \mathbf{z}\mapsto A\cdot \mathbf{z}
$$
is an element of $Aut_{\mathbb{C}}(\mathbb{B}_{\mathbb{C}}^{n+1})$.  Secondly, for any fixed $\mathbf{b}\in \mathbb{B}^{n+1}_{\mathbb{C}}$ there is an element $\Phi_{\mathbf{b}}\in Aut_{\mathbb{C}}(\mathbb{B}_{\mathbb{C}}^{n+1})$ given by
\begin{align*}
\Phi_{\mathbf{b}}:\mathbf{z}\mapsto \sqrt{1-|\mathbf{b}|^2} \frac{\mathbf{z}}{1+ \bar{\mathbf{b}}\cdot \mathbf{z}}+\frac{1}{1+\sqrt{1-|\mathbf{b}|^2}}\left(1+\frac{\sqrt{1-|\mathbf{b}|^2}}{1+\bar{\mathbf{b}}\cdot \mathbf{z}} \right) \mathbf{b}.
\end{align*}
This map can also be expressed as
\begin{align*}
\Phi_{\mathbf{b}}(\mathbf{z})	&=\sqrt{1-|\mathbf{b}|^2} \frac{\mathbf{z}+\mathbf{b}}{1+ \bar{\mathbf{b}}\cdot \mathbf{z}}+\frac{1}{1+\sqrt{1-|\mathbf{b}|^2}}\frac{|\mathbf{b}|^2+\bar{\mathbf{b}}\cdot {\mathbf{z}}}{1+\bar{\mathbf{b}}\cdot \mathbf{z}}  \mathbf{b}.\\
\end{align*}
Using the identification $\mathbb{B}_{\mathbb{C}}^{n+1} \simeq \mathbb{B}^{2n+2}$ we may also think  of $\Phi_{\mathbf{b}}$ as a self-diffeomorphism of $\mathbb{B}^{2n+2}$ that is a $J_{\Real}$-holomorphic automorphism of $\mathbb{B}^{2n+2}$, in the sense that
$$
J_{\Real}\circ D_p \Phi =D_p\Phi\circ J_{\Real}.
$$
We identify  $Aut_{\mathbb{C}}(\mathbb{B}_{\mathbb{C}}^{n+1})$ with $Aut_{J}(\mathbb{B}^{2n+2})$, the set of $J_{\Real}$-holomorphic automorphisms of the disk.

\begin{prop}\label{PhibProp}
 For $\mathbf{b}=\mathbf{b}_1+i \mathbf{b}_2\in \mathbb{B}_{\mathbb{C}}^{n+1}$, the map $\Phi_{\mathbf{b}}$ satisfies:
 \begin{enumerate}
 	\item It extends to a $J_{\Real}$-holomorphic diffeomorphism  $\bar{\Phi}_{\mathbf{b}}:\bar{\mathbb{B}}^{2n+2}\to \bar{\mathbb{B}}^{2n+2}$.
 	\item On $\partial \mathbb{B}^{2n+2}$, 
 	$$
 	\bar{\Phi}_{\mathbf{b}}^* dr= W_{\mathbf{b}} dr \mbox{ and } \bar{\Phi}_{\mathbf{b}}^* \theta= W_{\mathbf{b}} \theta$$
 	where
 	$$ 
 	W_{\mathbf{b}}(\mathbf{x}, \mathbf{y})= \frac{1-|\mathbf{b}_1|^2-|\mathbf{b}_2|^2}{(1+\mathbf{b}_1\cdot \mathbf{x}+\mathbf{b}_2\cdot \mathbf{y})^2+(\mathbf{b}_2\cdot \mathbf{x}-\mathbf{b}_1\cdot \mathbf{y})^2}.
 	$$
 	\item On $\partial \mathbb{B}^{2n+2}$,
 	$$
 	\bar{\Phi}_{\mathbf{b}}^* d \theta = W_{\mathbf{b}} d\theta+ 2 (W_{\mathbf{b}}-W_{\mathbf{b}}^2)r dr \wedge \theta+ dW_{\mathbf{b}}\wedge \theta+dr\wedge (dW_{\mathbf{b}}\circ J_{\Real}).
 	$$
 \end{enumerate}
\end{prop}
\begin{proof}
	The fact that $\Phi_{\mathbf{b}}$ extends smoothly can be readily seen from the formula and it is clear that the extended map is $J_{\Real}$-holomorphic and smoothly invertible. A straightforward computation gives
\begin{equation}\label{NormPhibEq}
	|\bar{\Phi}_{\mathbf{b}}(\mathbf{z})|^2=1-\frac{(1-|\mathbf{b}|^2)(1-|\mathbf{z}|^2)}{|1+\bar{\mathbf{b}}\cdot \mathbf{z}|^2}.
\end{equation}
	With $\mathbf{z}=\mathbf{x}+i\mathbf{y}$, it is convenient to write
	$$
	W_{\mathbf{b}}(\mathbf{z})=(1-|\mathbf{b}|^2)|1+\bar{\mathbf{b}}\cdot \mathbf{z}|^{-2}=W_{\mathbf{b}}(\mathbf{x}, \mathbf{y}).
	$$
	By inspection, $W_{\mathbf{b}}$ is a smooth and non-zero function on $\bar{\mathbb{B}}^{2n+2}$.
We write \eqref{NormPhibEq} as
	$$
	r^2\circ \bar{\Phi}_{\mathbf{b}}=1-(1-r^2)W_{\mathbf{b}}.
	$$
	Hence,
	\begin{align*}
		\left(2r\circ \bar{\Phi}_{\mathbf{b}}\right)\bar{\Phi}_{\mathbf{b}}^*dr &=\bar{\Phi}_{\mathbf{b}}^*(2r dr) =\bar{\Phi}_{\mathbf{b}}^*d(r^2)=d(r^2\circ\bar{\Phi}_{\mathbf{b}})\\
		& =d( 1-(1-r^2)W_{\mathbf{b}} )= -(1-r^2) dW_{\mathbf{b}}+2W_{\mathbf{b}} rdr.
	\end{align*}
  Combining this with the $J_{\Real}$-holomorphicity of $\bar{\Phi}_{\mathbf{b}}$ yields 
	\begin{align*}
		\bar{\Phi}_{\mathbf{b}}^*(r^2\theta) &= \bar{\Phi}_{\mathbf{b}}^*\left( rdr \circ J_{\Real}\right) =(r\circ \bar{\Phi}_{\mathbf{b}}) ( \bar{\Phi}_{\mathbf{b}}^* dr) \circ J_{\Real}= W_{\mathbf{b}}  r^2\theta-\frac{1-r^2}{2}dW_{\mathbf{b}}\circ J_{\Real}.
	\end{align*}
The above computations combined imply the second claim.
We further compute, 
	\begin{align*}
		\bar{\Phi}_{\mathbf{b}}^* r^2d \theta &= \bar{\Phi}_{\mathbf{b}}^* \left(d( r^2  \theta)-2rdr\wedge \theta\right) =d(\bar{\Phi}_{\mathbf{b}}^*r^2 \theta)- 2\bar{\Phi}_{\mathbf{b}}^* dr \wedge \bar{\Phi}_{\mathbf{b}}^*(r\theta)\\
		&= d W_\mathbf{b} \wedge r^2 \theta+W_\mathbf{b} d(r^2 \theta) +rdr\wedge dW_\mathbf{b} \circ J_{\Real} -\frac{1-r^2}{2} d(dW_\mathbf{b}\circ J_{\Real})\\
		&-2W_{\mathbf{b}}^2 \frac{r^2}{r^2\circ \bar{\Phi}_{\mathbf{b}}} rdr \wedge \theta+\frac{r(1-r^2)}{r\circ \bar{\Phi}_{\mathbf{b}}}\left( dW_{\mathbf{b}}\wedge r \theta+dr\wedge dW_{\mathbf{b}}\circ J_{\Real}\right)\\
		&-\frac{(1-r^2)^2}{2 r^2\circ  \bar{\Phi}_{\mathbf{b}}}dW_{\mathbf{b}}\wedge  dW_{\mathbf{b}}\circ J_{\Real}.
	\end{align*}
	On $\partial \mathbb{B}^{2n+2}$, this simplifies to
	\begin{align*}
		\bar{\Phi}_{\mathbf{b}}^*d \theta =W_{\mathbf{b}} d\theta+ 2 (W_{\mathbf{b}}-W_{\mathbf{b}}^2)r dr \wedge \theta+ dW_{\mathbf{b}}\wedge \theta+dr\wedge (dW_{\mathbf{b}}\circ J_{\Real}),
	\end{align*}
	which verifies the third claim.

\end{proof}

Every element $\Phi\in Aut_{\mathbb{C}}(\mathbb{B}_{\mathbb{C}}^{n+1})$ satisfies 
\begin{equation}\label{StructureAutCEqn}
	\Phi=\Phi_A\circ \Phi_{\mathbf{b}}
\end{equation}
 for some $A\in \mathbf{U}(n+1)$ and $\mathbf{b}\in \mathbb{B}_{\mathbb{C}}^{n+1}$. 
It follows from \eqref{StructureAutCEqn} and Proposition \ref{PhibProp} that every element  $\Phi\in Aut_{J}(\mathbb{B}^{2n+2})$ extends smoothly to a map $\bar{\Phi}: \bar{\mathbb{B}}^{2n+2}\to \bar{\mathbb{B}}^{2n+2}$.  We write $Aut_{J}(\bar{\mathbb{B}}^{2n+2})$ for the group of extended maps. The maps $\bar{\Phi}\in Aut_{J}(\bar{\mathbb{B}}^{2n+2})$ have the additional property that $\Psi=\bar{\Phi}|_{\mathbb{S}^{2n+1}}$ is a diffeomorphism of $\mathbb{S}^{2n+1}$ to itself, which is a \emph{CR-automorphism} of $\mathbb{S}^{2n+1}$, that is,
$$
D_p{\Psi}(\mathcal{H}_p)=\mathcal{H}_{{\Psi}(p)}  \mbox{ and }
J_{\mathbb{S}}\circ D_p{\Psi}=D{\Psi}_p \circ J_{\mathbb{S}}.
$$
Let us denote the set of such maps  by $Aut_{CR}(\mathbb{S}^{2n+1})$.  For $\mathbf{b}=\mathbf{b}_1+i \mathbf{b}_2\in \mathbb{B}_{\mathbb{C}}^{n+1}$, let $\Psi_{\mathbf{b}}\in Aut_{CR}(\mathbb{S}^{2n+1})$ be,  the restriction of $\bar{\Phi}_{\mathbf{b}}\in Aut_J(\mathbb{B}^{2n+2})$.  For $A\in \mathbf{U}(n+1)$ define $\Psi_A\in Aut_{CR}(\mathbb{S}^{2n+1})$ in the same manner.

\begin{prop}\label{CRAutomorphismsThm}
The following properties of $Aut_{CR}(\mathbb{S}^{2n+1})$ hold:
\begin{enumerate}
	\item Every element of $Aut_{CR}(\mathbb{S}^{2n+1})$ is the restriction of a unique element of $Aut_{J}(\bar{\mathbb{B}}^{2n+2})$ and in particular, for every $\Psi\in Aut_{CR}(\mathbb{S}^{2n+1})$ there is an $A\in \mathbf{U}(n+1)$ and $\mathbf{b}=\mathbf{b}_1+i \mathbf{b}_2$ such that $\Psi=\Psi_A\circ \Psi_{\mathbf{b}}$.
	\item The elements $\Psi\in Aut_{CR}(\mathbb{S}^{2n+1})$ are contactomorphisms, in fact for $\Psi=\Psi_A\circ \Psi_{\mathbf{b}}$ one has ${\Psi}^* \hat{\theta}=W_{\mathbf{b}} \hat{\theta}$.
	\item An  element $\Psi=\Psi_A\circ \Psi_{\mathbf{b}}\in Aut_{CR}(\mathbb{S}^{2n+1})$ acts on the metric $g_{\mathbb{S}}$ by
\begin{align*}
	{\Psi}^* g_{\mathbb{S}}&= W_{\mathbf{b}}\left( g_{\mathbb{S}} + \hat{\omega}_{\mathbf{b}} \cdot \hat{\theta}+ \hat{\theta}\cdot \hat{\omega}_{\mathbf{b}}+ \partial_r \log W_{\mathbf{b}} \hat{\theta}^2 \right)\\
	&= W_{\mathbf{b}}\left( \hat{\eta}+ \hat{\omega}_{\mathbf{b}} \cdot \hat{\theta}+ \hat{\theta}\cdot \hat{\omega}_{\mathbf{b}}+ \left( W_{\mathbf{b}} +\frac{1}{4}|\nabla^{\mathcal{H}}\log W_{\mathbf{b}}|^2_{\mathbb{S}}\right) \hat{\theta}^2\right)
\end{align*}
	where $\hat{\omega}_{\mathbf{b}}$ is a one form on $\mathbb{S}^{2n+1}$ given by
	$$
	\hat{\omega}_{\mathbf{b}}=\frac{1}{2} d\log W_{\mathbf{b}}\circ J_{\mathbb{S}}=-\frac{1}{2}g_{\mathbb{S}}(J_{\mathbb{S}}(\nabla^{\mathcal{H}}W_{\mathbf{b}}), \cdot).
	$$
\end{enumerate}
\end{prop}
\begin{rem}
Clearly,  $(\Psi^* g_{\mathbb{S}})|_{\mathcal{H}}= W_{\mathbf{b}} g_{\mathbb{S}}|_{ \mathcal{H}}$ for some $\mathbf{b}$ and so on the horizontal distribution the elements of $Aut_{CR}(\mathbb{S}^{2n+1})$ act conformally.  The deformations producing $g_{\mathbb{S}}^{\mathbf{b}}=\Psi^*_{\mathbb{b}} g_{\mathbb{S}}$ are special cases of more general deformations of contact Riemannian manifolds studied by Tanno in \cite[Section 9]{Tanno1989} where they were called \emph{gauge transformations of contact Riemannian structures}.
\end{rem}
\begin{proof}
The first claim is a standard fact -- see \cite[Lemma 1.1]{PaulYangEtAl} and \cite{chernRealHypersurfacesComplex1974}.  The second  is an immediate consequence of Proposition \ref{PhibProp} and the fact that $\Psi_A^*\hat{\theta}=\hat{\theta}$.

By \eqref{PhibProp}, on $\mathbb{S}^{2n+1}=\partial \mathbb{B}^{2n+2}$,  one has
\begin{align*}
	 \Psi_{\mathbf{b}}^* \eta &= 		
		 W_{\mathbf{b}} \eta +(W_{\mathbf{b}}-W_{\mathbf{b}}^2 ) (dr^2+\theta^2)+\frac{1}{2}d W_{\mathbf{b}} \cdot dr \\
		 &+\frac{1}{2} \theta\cdot(d W_{\mathbf{b}} \circ J_{\Real}) +\frac{1}{2} dr \cdot d W_{\mathbf{b}}+\frac{1}{2}(d W_{\mathbf{b}} \circ J_{\Real})\cdot \theta.
\end{align*}
Hence, if $\hat{i}:\mathbb{S}^{2n+1}\to \Real^{2n+1}$ is the usual inclusion, then $\hat{\eta}=\hat{i}^* \eta$ and
\begin{align*}
\Psi_{\mathbf{b}}^* \hat{\eta}&= W_{\mathbf{b}} \hat{\eta}+(W_{\mathbf{b}}-W_{\mathbf{b}}^2 )\hat{\theta}^2+ \frac{1}{2}\hat{\theta} \cdot \hat{i}^*(dW_{\mathbf{b}} \circ J_{\Real}) + \frac{1}{2} \hat{i}^*(dW_{\mathbf{b}} \circ J_{\Real})\cdot \hat{\theta}\\
&= W_{\mathbf{b}} \hat{\eta}+(W_{\mathbf{b}}-W_{\mathbf{b}}^2+\frac{\partial W_{\mathbf{b}} }{\partial r} )\hat{\theta}^2+ \frac{1}{2}\hat{\theta} \cdot d\hat{W}_{\mathbf{b}} \circ J_{\mathbb{S}}) + \frac{1}{2} (d\hat{W}_{\mathbf{b}} \circ J_{\mathbb{S}})\cdot \hat{\theta}
\end{align*} 
where $\hat{W}_{\mathbf{b}}=W_{\mathbf{b}}\circ \hat{i}$.
As $g_{\mathbb{S}}=\hat{\theta}^2+ \hat{\eta}$, it follows that
$$
{\Psi}_{\mathbb{b}}^*g_{\mathbb{S}}= W_{\mathbf{b}} \left( g_{\mathbb{S}}+  \hat{\omega}_{\mathbf{b}} \cdot \hat{\theta}+\hat{\theta}\cdot \hat{\omega}_{\mathbf{b}}+\partial_r \log W_{\mathbf{b}}\hat{\theta}^2 \right).
$$
A straightforward computation shows that the following identity holds on $\mathbb{S}^{2n+1}$
\begin{align*}
	\partial_r W_{\mathbf{b}}= W_{\mathbf{b}}^2-W_\mathbf{b} +\frac{1}{4}\frac{|\nabla^{\mathcal{H}} W_{\mathbf{b}}|^2}{W_{\mathbf{b}}}.
\end{align*}
The final claim then follows from the above computations, the first claim, and the fact that $\Psi_A$ is a $g_{\mathbb{S}}$-isometry.
\end{proof}

\subsection{CR-Volume}\label{CRVolSec}
We introduce the $CR$-volume for horizontal and Legendrian submanifolds of $\mathbb{S}^{2n+1}$ -- we establish some basic properties of this functional in \cite{BernBhattCR}.  Recall, if $\Gamma\subset \mathbb{S}^{2n+1}$ is an $m$-dimensional submanifold, then $\Gamma$ is \emph{horizontal} if $T_p \Sigma \subset \mathcal{H}_p$ for all $p\in \Sigma$, where $\mathcal{H}_p$ is as defined in (\ref{Hp}). The properties of CR-manifolds ensure that $m\leq n$.  When $m=n$ we say that $\Gamma$ is \emph{Legendrian}.

\begin{definition} \label{CRVol}
The \emph{CR-volume} of $\Gamma\subset \mathbb{S}^{2n+1}$,  an $m$-dimensional horizontal submanifold, is:
$$
\lambda_{CR}[\Gamma]=\sup_{{\Psi}\in Aut_{CR}(\mathbb{S}^{2n+1})}|{\Psi}(\Gamma)|_{\mathbb{S}}.
$$
  \end{definition}

Observe that elements ${\Psi}\in Aut_{CR}(\mathbb{S}^{2n+1})$ can be factored as $
\Psi=\Psi_{A}\circ\Psi_\mathbf{b}
$
where $A\in \mathbf{U}(n+1)$ and $\mathbf{b}\in  \mathbb{B}_{\mathbb{C}}^{n+1}$. Using the facts that ${\Psi}_{A}$ is an isometry of $g_{\mathbb{S}}$, $\Gamma$ is horizontal, along with Proposition \ref{CRAutomorphismsThm} we obtain
\begin{align*}
	\lambda_{CR}[\Gamma]&=\sup_{\mathbf{b}\in \mathbb{B}_{\mathbb{C}}^{n+1}} |\Psi_{\mathbf{b}}(\Gamma)|_{\mathbb{S}}=\sup_{\mathbf{b}\in \mathbb{B}_{\mathbb{C}}^{n+1}} \int_{\Gamma} W_{\mathbf{b}}^{m/2}(p) dV_{\Gamma}(p)\\
	&=\sup_{\mathbf{b}\in \mathbb{B}_{\mathbb{C}}^{n+1}} \int_{\Gamma} \frac{(1-|\mathbf{b}|^2)^{m/2}}{|1+\bar{\mathbf{b}}\cdot \mathbf{z}(p)|^{m}} dV_{\Gamma}(p).
\end{align*}

\section{Complex hyperbolic space and its compactifications}\label{CompactSec}
In this section, we establish some properties of complex hyperbolic space, $\mathbb{CH}^{n+1}$, and its submanifolds.  Recall,  $\mathbb{CH}^{n+1}$ is the  simply connected (complex) space form with constant holomorphic sectional curvature.  We think of it as a $(2n+2)$-dimensional  K\"ahler-Einstein manifold with Riemannian metric $g_{\mathbb{CH}}$ and the integrable  (almost) complex structure $J_{\mathbb{CH}}$. We refer to \cite{GoldmanCHBook} for background on this space, though, however, unlike in \cite{GoldmanCHBook}, we adopt the convention that the sectional curvatures of $g_{\mathbb{CH}}$ lie between $-4$ and $-1$. 

\subsection{Bergman metric and the Bergman compactification of submanifolds}\label{CompactSec_1}

A standard model of complex hyperbolic space is the \emph{Bergman model}. Here the underlying manifold is $\mathbb{B}^{2n+2}\simeq \mathbb{B}_{\mathbb{C}}^{n+1}$ with metric
\begin{align*}
g_{B}&= \frac{1}{1-r^2} \sum_{j=1}^{n+1} dz_j d\bar{z}_j +\frac{1}{(1-r^2)^2} \sum_{j,k=1}^{n+1} z_j \bar{z}_k dz_k d\bar{z}_j \\
&=
\frac{1}{(1-r^2)^2} dr^2+\frac{r^2}{(1-r^2)^2} \theta^2 + \frac{r^2}{1-r^2} \eta.
\end{align*}
One of the advantages of this model is that in it the complex structure, $J_B$, satisfies $J_B=J_{\Real}$. In particular,  $Aut_{g_B}^+(\mathbb{B}^{2n+2})$, the orientation preserving isometries of $g_B$, is identified with $Aut_{J}(\mathbb{B}^{2n+2})$. Finally, the corresponding symplectic form is 
$$
\omega_{B}=\frac{r^2}{1-r^2} \frac{1}{2} d\theta+\frac{1}{(1-r^2)^2} r dr \wedge \theta=\frac{1}{1-r^2} \omega_\Real+\frac{r^2}{(1-r^2)^2} r dr \wedge \theta.
$$

For any $p\in \mathbb{CH}^{n+1}$ there is a diffeomorphism $\Upsilon_{p}: \mathbb{CH}^{n+1}\to \mathbb{B}^{2n+2}$ such that $\Upsilon_{p}(p)=\mathbf{0}$ and $\Upsilon_{p}^* g_B=g_{\mathbb{CH}}$.  Moreover, this map is a biholomorphism and is unique up to post-composition with an element of $\mathbf{U}(n+1)$.    By making the identification with $\bar{\mathbb{B}}^{2n+2}$, this leads naturally to a compactification, $\overline{\mathbb{CH}}^{n+1}$, of complex hyperbolic space, which we call a \emph{Bergman compactification}.  Observe, in this case,  the \emph{ideal boundary} $\partial_\infty \mathbb{CH}^{n+1}$ is identified with $\mathbb{S}^{2n+1}=\partial \mathbb{B}^{2n+2}$.  Different choices of $p$ and the corresponding $\Upsilon_p$ give different compactifications, but they induce equivalent structures as manifolds with boundary.  They also endow $\partial_{\infty} \mathbb{CH}^{n+1}$ with a canonical $CR$-structure, though only an equivalence class of Sasaki structures. 

We use Bergman compactifications to define asymptotic properties of submanifolds $\Sigma \subset \mathbb{CH}^{n+1}$.  To that end,  for an $m$-dimensional submanifold $\Sigma \subset \mathbb{CH}^{n+1}$, a point $p\in \mathbb{CH}^{n+1}$, and a Bergman compactification $\Upsilon_p: \mathbb{CH}^{n+1}\to \mathbb{B}^{2n+2}$, let
$$
\Sigma_p= \overline{\Upsilon_p(\Sigma)}\subset \bar{\mathbb{B}}^{2n+2}.
$$
We call $\Sigma_p$ a \emph{Bergman compactification} of $\Sigma$. 
\begin{definition} \label{def_asymreg}
    Suppose that $\Sigma$ and $\Sigma_p$ are as above:
    \begin{enumerate}
    	\item For any $l\geq 1$, $\Sigma$ is \emph{$C^l$-asymptotically regular} if $\Sigma_p$ is a $C^l$-regular submanifold with boundary and $\Sigma_p$ meets $\partial \mathbb{B}^{2n+2}$ transversally;
    	\item If, in addition, $\partial \Sigma_p$ is a horizontal submanifold relative to the usual contact structure on $\mathbb{S}^{2n+1}=\partial \mathbb{B}^{2n+2}$, i.e., for all $q\in \partial \Sigma_p$,  $T_q \partial \Sigma_p\subset \mathcal{H}_q=\ker \hat{\theta}_q$,  then  $\Sigma$ is \emph{asymptotically horizontal}; 
    	\item If, in addition, for all $q\in \partial \Sigma_p$, $T_q \Sigma_p$ is orthogonal to $\mathbf{T}$ , i.e., $T_q \Sigma_p \subset \ker \theta_q$, then  $\Sigma$ is \emph{strongly asymptotically horizontal}.
    \end{enumerate}   
 Note (2) and (3) can only hold when $m\leq n+1$.  When $m=n+1$ the term \emph{horizontal} is replaced by \emph{Legendrian}.
 \end{definition}

 If $\Sigma_{p'}$ is another choice of a  Bergman compactification of $\Sigma$, then $\Sigma_p=\bar{\Phi}(\Sigma_{p'})$ for some $\bar{\Phi} \in Aut_{J}(\bar{\mathbb{B}}^{2n+2})$.  It follows from Propositions \ref{PhibProp} and \ref{CRAutomorphismsThm} that Definition \ref{def_asymreg} is independent of choices.   In particular,  we may think of a $C^l$-asymptotically regular submanifold $\Sigma$, as having a well defined $C^l$-regular \emph{ideal boundary}, $\partial_\infty \Sigma \subset \partial_\infty \mathbb{CH}^{n+1}$.    The submanifold $\Sigma$ is asymptotically horizontal precisely when $\partial_\infty \Sigma$ satisfies this property with respect to the CR structure of $\partial_\infty\mathbb{CH}^{n+1}$.  

The notion of being  asymptomatically horizontal is natural as it automatically holds for isotropic submanifolds of sufficient asymptotic regularity.
\begin{lem}\label{IsotropIsHorizLem}
	Let $\Sigma\subset \mathbb{CH}^{n+1}$ be a $C^1$-asymptotically regular $m$-dimensional submanifold.  If $m\geq 2$ and $\Sigma$ is isotropic, then it is asymptotically  horizontal.  
\end{lem}
\begin{proof}
Let $\Sigma_p$ be a Bergman compactification of $\Sigma$.  Clearly, $\Sigma $ is isotropic if and only if $\Sigma_p$ is isotropic.
  Hence,  for $X,Y\in T_q \Sigma_p$, one has $\omega_B(X,Y)= g_B(X, J_{B}(Y))=0$. Let us denote by $\mathbf{T}^\top$, the tangential component, with respect to $g_{\Real}$, of $\mathbf{T}$ along $\Sigma_p$.  Using $J_{\Real}=J_B$ we have,
  $$
  g_B(\mathbf{T}^\top, J_{\Real}(\mathbf{X}^\top))=0.
  $$
  As $\Sigma_p$ is $C^1$ up to the boundary and meets $\partial \mathbb{B}^{2n+2}$ transversally, 
  $$
  \mathbf{X}^\top =\alpha \mathbf{X}+\gamma \mathbf{T}+\mathbf{v}
  $$
  where $\mathbf{v}$ is $g_{\Real}$-orthogonal to $\mathbf{X}$ and $\mathbf{T}$.  Likewise, one has 
  $$
  \mathbf{T}^\top =\beta \mathbf{T}+\delta \mathbf{X}+\mathbf{w}
   $$
   where $\mathbf{w}$ is $g_{\Real}$-orthogonal to $\mathbf{X}$ and $\mathbf{T}$.  On the boundary, we have
   $$
   \beta=|\mathbf{T}^\top|^2 \mbox{ and } \alpha=|\mathbf{X}^\top|^2 \neq 0
   $$
   while
   $$
   \gamma= g_{\Real}(\mathbf{X}^\top, \mathbf{T})=g_{\Real}(\mathbf{X}^\top, \mathbf{T}^\top)  =g_{\Real}(\mathbf{X}, \mathbf{T}^\top)=\delta.
   $$
 As $\mathbf{w}$ and ${J}_{\Real}(\mathbf{v})$ are orthogonal to $\mathbf{T}$, one computes
  \begin{align*}
    0&=  g_B(\mathbf{T}^\top, J_{\Real}(\mathbf{X}^\top))= g_B(\beta \mathbf{T}+\delta \mathbf{X}+\mathbf{w}, -\alpha \mathbf{T}+\gamma \mathbf{X}+J_{\Real}(\mathbf{v}))\\
    &=(-\alpha \beta+\gamma \delta)\frac{r^2}{(1-r^2)^2}+\frac{r^2}{1-r^2} \eta(\mathbf{w}, J_{\Real}(\mathbf{v})).
   \end{align*}
   Near the boundary, this gives the expansion:
   $$
   0=\frac{g_{\Real}(\mathbf{X}^\top, \mathbf{T}^\top)^2-|\mathbf{T}^\top|^2|\mathbf{X}^\top|^2}{4(1-r)^2}+o((1-r)^{-2}).
   $$
  The Cauchy-Schwarz inequality and the fact that $\mathbf{X}^\top\neq 0$ on the boundary,  implies $\mathbf{T}^\top= b\mathbf{X}^\top$.  As $\mathbf{X}^\top$ is $g_{\Real}$-orthogonal to $\partial \Sigma_p\subset \mathbb{S}^{2n+1}$,  the same is true of $\mathbf{T}^\top$. Hence, $\partial \Sigma_p$ is orthogonal to $\mathbf{T}$ and so it is horizontal.
\end{proof}

\subsection{Modified Bergman compactification}\label{sec_bergman} While the Bergman compactification is well-adapted to the complex geometry of $\mathbb{CH}^{n+1}$, it seems less satisfactory for studying the asymptotic regularity of minimal submanifolds;  this is apparent in the examples of \cite{castroMinimalLagrangianSubmanifolds2002}. 
Therefore, it is convenient to introduce a related compactification, which possesses certain computational features that make it similar to the usual conformal compactification of hyperbolic space.

To begin, consider the radial function $s:\mathbb{B}^{2n+2}\to [0,1]$ defined by
$$
s=\frac{r}{1+\sqrt{1-r^2}} \mbox{ or, equivalently, by }r=\frac{2s}{1+s^2}. 
$$
Observe that $r$ extends smoothly with derivative zero to $s=1$ while $s$ extends only as a $\frac{1}{2}$-H\"older continuous function to $r=1$.
Using
$$
1-s^2=2\left(1-\frac{1}{1+\sqrt{1-r^2}}\right)=2 \frac{\sqrt{1-r^2}}{1+\sqrt{1-r^2}} \mbox{ and }1-r^2=\frac{(1-s^2)^2}{(1+s^2)^2},
$$
one computes that
$$
g_{B}=\frac{4}{(1-s^2)^2}( ds^2+s^2 \theta^2+s^2\eta)+\frac{16 s^4}{(1-s^2)^4} \theta^2=g_{P}+\frac{16 s^4}{(1-s^2)^4} \theta^2
$$
where $g_P$ is the Poincar\'{e} metric on $\mathbb{B}^{2n+2}$ of constant curvature $-1$.
Likewise,
$$
\omega_{B}= \frac{4s^2}{(1-s^2)^2}\frac{1}{2} d\theta +\frac{4(1+s^2)}{(1-s^2)^3} s ds \wedge \theta.
$$

Now consider the diffeomorphism $\mathcal{S}: \mathbb{B}^{2n+2}\to \mathbb{B}^{2n+2}$ given by
$$
\mathcal{S}: \mathbf{z}\mapsto \frac{1}{1+\sqrt{1-|\mathbf{z}|^2}} \mathbf{z} \mbox{ with inverse } \mathcal{S}^{-1}: \mathbf{z}\mapsto  \frac{2}{1+|\mathbf{z}|^2}\mathbf{z}.
$$
This map extends to a $\frac{1}{2}$-H\"{o}lder continuous, but not smooth, homeomorphism, $\bar{\mathcal{S}}$, from the closed ball $\bar{\mathbb{B}}^{2n+2}$ to itself.  Clearly, $s(p)=r(\mathcal{S}(p))$
and so
$$
\mathcal{S}^* dr=ds=\frac{1}{\sqrt{1-r^2}(1+\sqrt{1-r^2})} dr \mbox{ and }(\mathcal{S}^{-1})^* ds =dr= \frac{2(1-s^2)}{(1+s^2)^2} ds.
$$
Moreover,
$$
\mathcal{S}^* \theta=\theta \mbox{ and } \mathcal{S}^* \eta=\eta.
$$
Hence, if we define a metric on $\mathbb{B}^{2n+2}$ by
$$
{g}_{\tilde{B}}=(\mathcal{S}^{-1})^* g_B=\frac{4}{(1-r^2)^2} g_E + \frac{16 r^4}{(1-r^2)^4}\theta^2= g_P + \frac{16 r^4}{(1-r^2)^4}\theta^2,
$$
then $\mathcal{S}^* {g}_{\tilde{B}}=g_B$.  Here ${g}_{\tilde{B}}$ is the \emph{modified Bergman metric} which is obtained from the Poincar\'{e} metric in a particularly simple manner.  The corresponding symplectic form, $\omega_{\tilde{B}}$ satisfies $\mathcal{S}^* \omega_{\tilde{B}}=\omega_B$ and the compatible almost complex structure is
$$
J_{\tilde{B}}=J_{\Real}+ 2r \left( \frac{1}{1-r^2} \mathbf{X}\otimes r\theta-\frac{1}{1+r^2} \mathbf{T}\otimes dr\right).
$$
A consequence is that $Aut_{{g}_{\tilde{B}}}^+(\mathbb{B}^{2n+2})$, the orientation preserving isometries of ${g}_{\tilde{B}}$ are not holomorphic with respect to the usual complex structure of the ball.  However, as every element $\tilde{\Phi}\in Aut_{g_{\tilde{B}}}^+(\mathbb{B}^{2n+2})$ is of the form
$$
\tilde{\Phi}= \mathcal{S} \circ \Phi \circ \mathcal{S}^{-1}
$$
for a unique $\Phi \in Aut_{g_B}^+(\mathbb{B}^{2n+2})$, it follows that $\tilde{\Phi}$ extends smoothly to $\bar{\mathbb{B}}^{2n+2}$ and induces the same map in $Aut_{CR}(\mathbb{S}^{2n+1})$  as $\bar{\Phi}$.

We now use $\mathcal{S}$ to define a modified form of the Bergman compactification.  
Fix $p\in \mathbb{CH}^{n+1}$ and let $ \Upsilon_p:\mathbb{CH}^{n+1}\to \mathbb{B}^{2n+2}$ be the corresponding  choice of Bergman compactification. Let $\tilde{\Upsilon}_p: \mathbb{CH}^{n+1}\to \mathbb{B}^{n+1}$ be the map $\tilde{\Upsilon}_p=\mathcal{S}\circ \Upsilon_p$ so $\tilde{\Upsilon}_p^* \tilde{g}_B=g_{\mathbb{CH}}$ and $\tilde{\Upsilon}_p(p)=\mathbf{0}$.   For $\Sigma\subset \mathbb{CH}^{n+1}$, an $m$-dimensional submanifold,  set
$$
\tilde{\Sigma}_p=\bar{\mathcal{S}}(\Sigma_p)=\overline{\tilde{\Upsilon}_p(\Sigma)}\subset \bar{\mathbb{B}}^{2n+2},
$$ 
which we call a \emph{modified Bergman compactification} of $\Sigma$.
\begin{definition}\label{def_weakasym}
	Suppose that $\Sigma$ and $\tilde{\Sigma}_p$ are as above:
    \begin{enumerate}
    	\item  For any $l\geq 1$, $\Sigma$ is \emph{weakly $C^l$-asymptotically regular} if $\tilde{\Sigma}_p$ is a $C^l$-regular  submanifold with boundary in $\mathbb{B}^{2n+2}$ that meets $\partial \mathbb{B}^{2n+2}$ transversally;    
       \item If, in addition to (1), $\tilde{\Sigma}_p$ meets $\partial  \mathbb{B}^{2n+2}$ orthogonally then $\Sigma$ is \emph{asymptotically quasi-normal};
       	\item If, in addition to (1),  $\partial \tilde{\Sigma}_p$ is a horizontal submanifold of $\mathbb{S}^{2n+1}=\partial \mathbb{B}^{2n+2}$, then  $\Sigma$ is \emph{weakly asymptotically horizontal}. 
        \end{enumerate}
     For (3) to hold,   $m\leq n+1$ and when $m=n+1$ the term \emph{horizontal} in (3) is replaced by \emph{Legendrian}.
 \end{definition}

In the above definition, the independence of items (1) and (3)  from the choice of $p$ follow from the properties of $Aut_{g_{\tilde{B}}}^+(\mathbb{B}^{2n+2})$.  To establish the independence of item (2) and to relate Definition \ref{def_weakasym} to Definition \ref{def_asymreg} we use the following result. 
\begin{lem}\label{WeakAsy2StrongAsyLem}
Let $\Sigma\subset \mathbb{CH}^{n+1}$ be an $m$-dimensional submanifold and fix an $l\geq 1$.  
\begin{enumerate}
	\item If $\Sigma$ is {$C^l$-asymptotically regular}, then $\Sigma$ is weakly $C^l$-asymptotically regular and asymptotically quasi-normal.
    \item Conversely, if $\Sigma$ is weakly $C^l$-asymptotically regular and asymptotically quasi-normal, then  $\Sigma$ is {$C^1$-asymptotically regular}.
    \item If $\Sigma$ is {$C^1$-asymptotically regular}, then $\Sigma$ is  asymptotically horizontal if and only if $\Sigma$ is weakly horizontal.
    \item  If $\Sigma$ is weakly {$C^2$-asymptotically regular} and asymptotically quasi-normal,  then $\Sigma$ is strongly asymptotically horizontal if and only if $\Sigma$ is weakly horizontal and, for a modified Bergman compactification, $\tilde{\Sigma}_p$, 
    $$
    g_{\Real}(\mathbf{T},\mathbf{A}_{\tilde{\Sigma}_p}^\Real(\mathbf{X}^\top, \mathbf{X}^\top))=0 \mbox{ along $\partial \tilde{\Sigma}_p$}.
    $$
\end{enumerate}
\end{lem}
\begin{rem}
	We observe that the converse direction of (2) must be genuinely weaker when $l\geq 2$ as can be seen by the examples of \cite{castroMinimalLagrangianSubmanifolds2002}.
\end{rem}
\begin{proof}
Let $\Sigma_p=\overline{\Upsilon_{p}(\Sigma)}\subset \bar{\mathbb{B}}^{2n+2}$ be a Bergman compactification of $\Sigma$.
The fact that $\Sigma$ is $C^l$-asymptotically regular means the following: $\Sigma$ is a $C^l$- submanifold with boundary, $\Gamma=\partial \Sigma_p\subset \partial \mathbb{B}^{2n+2}$ is a $C^l$-submanifold, and $\Sigma_p$ meets $ \partial \mathbb{B}^{2n+2}$ transversally along $\Gamma$. Hence, there is a parametrization of $\Sigma_p$ in a neighborhood of $\Gamma$ given by a $C^l$ map
$$
\mathbf{F}: (1-\epsilon]\times \Gamma\to \Sigma_p \subset \bar{\mathbb{B}}^{2n+2}
$$
with the property that $|\mathbf{F}(\rho, q)|=\rho$.  By Taylor's theorem, we may write
$$
\mathbf{F}(\rho, q)= \mathbf{X}(q)+ \sum_{i=1}^l (1-\rho)^i \mathbf{a}_i(q)+(1-\rho)^{l}  \mathbf{f}(\rho,q)
$$
where $\mathbf{f}(1,q)=\mathbf{0}$.   The properties of $\Sigma_p$ ensure $g_{\Real}(\mathbf{a}_1(q), \mathbf{X}(q))<0$.
Thus,
\begin{align*}
\mathcal{S}(\mathbf{F}(\rho, q))&= \frac{1}{1+\sqrt{1-\rho^2}} \mathbf{F}(\rho, q)\\
&= \frac{1}{1+\sqrt{1-\rho^2}}  \mathbf{X}(q)+ \sum_{i=1}^l \frac{(1-\rho)^i}{1+\sqrt{1-\rho^2}} \mathbf{a}_i(q)+\frac{(1-\rho)^{l}}{1+\sqrt{1-\rho^2}}  \mathbf{f}(\rho,q). 	
\end{align*}
Setting 
$$
\tilde{\mathbf{F}}(\sigma, q)= \mathcal{S}\left(\mathbf{F}\left(\frac{2\sigma}{1+\sigma^2}, q\right)\right)
$$
one obtains that
$$
\tilde{\mathbf{F}}(\sigma, q)= \frac{1+\sigma^2}{2} \mathbf{X}(q)+ \frac{1}{2}\sum_{i=1}^l (1-\sigma)^{2i}(1+\sigma^2)^{i-1} \mathbf{a}_i(q)+\frac{1}{2}(1-\sigma)^{2l} \tilde{\mathbf{f}}(\sigma, q).
$$
It follows that, up to shrinking $\epsilon$, 
$$\tilde{\mathbf{F}}: (1-\epsilon,1]\times \Gamma\to \tilde{\Sigma}_p\subset \bar{\mathbb{B}}^{2n+2}$$
 is a $C^l$-embedding and so parametrizes $\tilde{\Sigma}_p=\bar{\mathcal{S}}(\Sigma_p)$ in a neighborhood of $\tilde{\Gamma}=\Gamma$.  In particular,  $\tilde{\Sigma}_p$ is a $C^l$-regular submanifold with boundary.  Moreover, 
$$
\partial_\sigma \tilde{\mathbf{F}}(1,q)= \mathbf{X}(q)
$$
and so we can conclude that $\tilde{\Sigma}_p$ meets $\partial \mathbb{B}^{2n+2}$ orthogonally.   This verifies (1).

In the converse direction, let $\tilde{\Sigma}_p$ be the appropriate modified Bergman compactification of $\Sigma$.  The hypotheses ensure that $\tilde{\Sigma}_p$ is a $C^l$-regular submanifold with boundary that meets $\partial \mathbb{B}^{2n+2}$ orthogonally.  This means that there
is a parametrization of $\tilde{\Sigma}_p$ in a neighborhood of $\Gamma$ by a $C^l$-embedding
$$
\tilde{\mathbf{G}}: (1-\epsilon,1]\times \Gamma\to \tilde{\Sigma}_p \subset \bar{\mathbb{B}}^{2n+2}
$$
with the property that $|\tilde{\mathbf{G}}(\sigma, q)|=\sigma$ and $\partial_\sigma \tilde{\mathbf{G}}(1,q)=\mathbf{X}(q)$.  By Taylor's theorem, 
$$
\tilde{\mathbf{G}}(\sigma, q)= \sigma\mathbf{X}(q)+ \frac{1}{2}(1-\sigma)^2 \mathbf{b}_2(q)+ (1-\sigma)^2 \mathbf{g}(\sigma, q)
$$
where $\mathbf{g}(1, q)=\mathbf{0}$ and we set $\mathbf{b}_2(q)=\mathbf{0}$ when $l=1$.  The hypotheses $|\tilde{\mathbf{G}}(\sigma, q)|=\sigma$ implies
$$
g_{\Real}(\mathbf{b}_2(q), \mathbf{X}(q))=0.
$$ 
One has
\begin{align*}
	\mathcal{S}^{-1}(\tilde{\mathbf{G}}(\sigma, q))&= \frac{2}{1+\sigma^2} \tilde{\mathbf{G}}(\sigma, q)\\
	 &=\frac{2\sigma}{1+\sigma^2} \mathbf{X}(q)+ \left(1-\frac{2\sigma}{1+\sigma^2}\right) \mathbf{b}_2(q)+2\left(1-\frac{2\sigma}{1+\sigma^2}\right) \mathbf{g}(\sigma, q).
\end{align*}
This is not an immersion at $\sigma=1$. However,  if we set
$$
\mathbf{G}(\rho, q)= \mathcal{S}^{-1}\left(\tilde{\mathbf{G}}\left(\frac{\rho}{1+\sqrt{1-\rho^2}}, q\right) \right),
$$
one obtains
$$
\mathbf{G}(\rho, q)= \rho \mathbf{X}(q)+(1-\rho)\left(\mathbf{b}_2(q)+\tilde{\mathbf{g}}(\rho,q)\right)
$$
where $ \tilde{\mathbf{g}}(1,q)=\mathbf{0}$. Up to shrinking $\epsilon$, this is readily checked to be a $C^1$ embedding on $(1-\epsilon, 1]\times \Gamma$ and so $\Sigma_p$ is a $C^1$-regular submanifold with boundary.  Moreover, $g_{\Real}(\partial_\rho \mathbf{G}(1, q), \mathbf{X}(q))=1$ and so $\Sigma_p$ meets $\partial \mathbb{B}^{2n+2}$ transversally. This verifies (2).  In addition, combined with (1) it also shows that being asymptotically quasi-normal is independent of the choice of modified Bergman compactification. Note that, when $l\geq 2$, unless $\mathbf{g}$ has appropriate parity at $\sigma=1$, there is a loss of regularity.

Having established (1) and (2),  (3) is an immediate consequence of the definitions.
Finally, by what has already been shown, the hypotheses of (4) imply that $\Sigma$ is $C^1$-asymptotically regular.  Using the parameterization, $\mathbf{F}$, from above with $l=1$, we see $\Sigma$ is strongly asymptotically horizontal, if and only if $g_{\Real}(\mathbf{a}_1(q), \mathbf{T})=0$ for all $q\in \Gamma$. It is not hard to see that this is equivalent to 
$$
g_{\Real}(\partial_\sigma^2 \tilde{\mathbf{F}}(1,q), \mathbf{T})=0,
$$
which can be readily checked to be equivalent to the geometric condition on $\tilde{\Sigma}_p$.
\end{proof}

\subsection{Second fundamental form and mean curvature in $\mathbb{CH}^{n+1}$}
On  $\mathbb{B}^{2n+2}$, let $h={g}_{\tilde{B}}$, be the modified Bergman metric,  $g=g_{P}$, be  the Poincar\'{e} metric, and set $\tau= \frac{4s^2}{(1-s^2)^2}\theta$.  Here  and in the following subsections we abuse notation and use $s$ instead of $r$ to emphasize that we are working with the modified Bergman metric.  As such,  $h$ is, in the sense of Appendix \ref{RankOneSec}, a rank one deformation of $g$ by  $\tau$ and  $\tau(X)=g(X, \mathbf{T})$.  We specialize the computations of Appendix \ref{CurvatureAppendix} to this case.

First observe that because $g_{\Real}$ and $g_P$ are conformal one has
$$
\mathbf{T}^{\hat{N}}= \mathbf{T}-\frac{1+|\mathbf{T}|_g^2}{1+|\mathbf{T}^\top|^2_g} \mathbf{T}^\top= \mathbf{T} -\frac{ (1+s)^2}{(1-s)^2+4|\mathbf{T}^\top|_{\Real}^2} \mathbf{T}^\top
$$
where we used the fact that $\mathbf{T}^N=\mathbf{T}^\perp$, i.e.,  the orthogonal component of $\mathbf{T}$ with respect to $\Sigma$ is the same for $g=g_P$ and $g_{\Real}$.
Likewise, for a vector field $\mathbf{V}$ along $\Sigma$
$$
\mathbf{V}^{\tilde{N}}=\mathbf{V}^N-g(\mathbf{V}^N, \mathbf{T}) \frac{\mathbf{T}^{{N}}}{|\mathbf{T}^{{N}}|_g^2}=\mathbf{V}^\perp-g_\Real(\mathbf{V}^\perp, \mathbf{T}) \frac{\mathbf{T}^{\perp}}{|\mathbf{T}^{\perp}|_\Real^2}.
$$
In particular,  $\mathbf{T}^{\tilde{N}}=\mathbf{0}$. 

\begin{prop}\label{2ndFFMCProp}
 Let $\Sigma\subset \mathbb{B}^{2n+2}$ be an $m$-dimensional submanifold.  The mean curvature of $\Sigma$ in $h=g_{\tilde{B}}$ and in  $g_{\Real}$ are related by:
 \begin{align*}
 	(\mathbf{H}_\Sigma^h)^{\tilde{N}} &= \frac{(1-s^2)^2}{4} (\mathbf{H}_\Sigma^\Real)^{\tilde{N}}-\frac{1-s^2}{2} (m+1) \mathbf{X}^{\tilde{N}}\\
 	&+\frac{(1-s^2)^2}{(1-s^2)^2 +4|\mathbf{T}^\top|_{\Real}^2} \left( \frac{1-s^2}{2} \mathbf{X}- \mathbf{A}_{\Sigma}^\Real(\mathbf{T}^\top, \mathbf{T}^\top) -2 J_{\Real}(\mathbf{T}^\top)\right)^{\tilde{N}}
 \end{align*}
and
 \begin{align*}
 	g_\Real(\mathbf{H}_\Sigma^h, \mathbf{T}^{\hat{N}})&=	\frac{(1-s^2)^2}{4} g_{\Real} (\mathbf{H}_\Sigma^\Real, \mathbf{T})+\frac{1-s^2}{2}\left( m+1+\frac{2}{1+s^2}\right) g_{\Real}(\mathbf{T}^\top, \mathbf{X}) \\
 	&+\frac{(1-s^2)^2}{(1-s^2)^2+4|\mathbf{T}^\top|_\Real^2}\left( \frac{1-s^2}{2} g_\Real(\mathbf{T}^\top, \mathbf{X}) - g_\Real(\mathbf{A}_\Sigma^\Real(\mathbf{T}^\top, \mathbf{T}^\top), \mathbf{T})\right). 
 \end{align*}
\end{prop}
\begin{proof}
Using $
\nabla^{{\Real}}_X \mathbf{T}=-J_{\Real}(X)$
and the formula for the connection of a conformally changed metric, one has
\begin{align*}
	\nabla_{Z}^{g} \mathbf{T}&= -J_{\Real}(Z)- Z \cdot \log (1-s^2) \mathbf{T} -\mathbf{T}\cdot \log (1-s^2) +g_\Real(Z,\mathbf{T}) \nabla^{\Real} \log (1-s^2)\\
	&=-J_{\Real}(Z)+ 2g_{\Real}(Z , \mathbf{X})  \frac{\mathbf{T}}{1-s^2}  -2g_\Real(Z,\mathbf{T}) \frac{\mathbf{X}}{1-s^2}.
\end{align*}
Using $g(X,J_{\Real}(Y))=-g(J_{\Real}(X), Y)$ and $J_{\Real}(\mathbf{T})=\mathbf{X} $ this yields,
\begin{align*}
	g(\nabla_{Z}^{g} \mathbf{T}, Y)&=g(Z, J_{\Real}(Y))+ g_P\left(\frac{2 g_{\Real}(Z, \mathbf{X})}{1-s^2} \mathbf{T}-\frac{2 g_{\Real}(Z, \mathbf{T})}{1-s^2}\mathbf{X}, {Y}\right)\\ 
	&=g(Z, J_{\Real}(Y))+ g_P\left(-\frac{2 g_{\Real}(Y, \mathbf{X})}{1-s^2} \mathbf{T}+\frac{2 g_{\Real}(Y, \mathbf{T})}{1-s^2}\mathbf{X}, Z\right).
\end{align*}
Hence, the tensor field, $\mathbf{a}$, from Proposition \ref{RankOne2ndFFProp} satisfies
\begin{align*}
	\mathbf{a}(\mathbf{T}^\top)=
\mathbf{a}(\mathbf{T}^\top)= J_{\Real}(\mathbf{T}^\top)-2g_{\Real}(\mathbf{T}^\top, \mathbf{X}) \frac{\mathbf{T}}{1-s^2} + 2|\mathbf{T}^\top|^2_\Real \frac{\mathbf{X}}{1-s^2}.
\end{align*}
By Corollary \ref{MCRankOneCor} and $(1-s^2)^2|\mathbf{T}^\top|_g^2=4|\mathbf{T}^\top|_\Real^2$, it follows that
\begin{align*}
(\mathbf{H}_\Sigma^h)^{\tilde{N}} &= (\mathbf{H}_\Sigma^g)^{\tilde{N}}- \frac{(\mathbf{A}_\Sigma^g(\mathbf{T}^\top, \mathbf{T}^\top))^{\tilde{N}} +2(J_{\Real}(\mathbf{T}^\top))^{\tilde{N}} +(1-s^2)|\mathbf{T}^\top|_g^2 \mathbf{X}^{\tilde{N}}}{1+|\mathbf{T}^\top|_g^2}.
\end{align*}
Likewise, the formula for the conformal change of the second fundamental form and metric implies that
$$
\mathbf{A}_\Sigma^g(X,Y)= \mathbf{A}_\Sigma^\Real(X,Y)-\frac{2g_{\Real}(X,Y)\mathbf{X}^\perp}{1-s^2}  \mbox{ and } \mathbf{H}_\Sigma^g=\frac{(1-s^2)^2}{4}\mathbf{H}_\Sigma^\Real -\frac{m(1-s^2)\mathbf{X}^\perp}{2}.
$$
Hence,
\begin{align*}
	(\mathbf{H}_\Sigma^h)^{\tilde{N}} &= \frac{(1-s^2)^2}{4} (\mathbf{H}_\Sigma^\Real)^{\tilde{N}}-\frac{m(1-s^2)\mathbf{X}^{\tilde{N}}}{2}\\
	&- \frac{(\mathbf{A}_\Sigma^\Real(\mathbf{T}^\top, \mathbf{T}^\top))^{\tilde{N}} +2(J_{\Real}(\mathbf{T}^\top))^{\tilde{N}}+\frac{1}{2}(1-s^2)|\mathbf{T}^\top|_g^2 \mathbf{X}^{\tilde{N}}   }{1+|\mathbf{T}^\top|_g^2}\\
	&=\frac{(1-s^2)^2}{4} (\mathbf{H}_\Sigma^\Real)^{\tilde{N}}-\frac{(1-s^2)}{2}(m+1)\mathbf{X}^{\tilde{N}} \\
	&+\frac{1}{1+|\mathbf{T}^\top|^2_g} \left( \frac{(1-s^2)}{2} \mathbf{X}-\mathbf{A}_\Sigma^\Real(\mathbf{T}^\top, \mathbf{T}^\top)-2J_{\Real}(\mathbf{T}^\top)\right)^{\tilde{N}}.
\end{align*}
Using  $(1-s^2)^2|\mathbf{T}^\top|_g^2=4|\mathbf{T}^\top|_\Real^2$  again, yields the first formula.

Using Corollary \ref{MCRankOneCor} and the formula for conformal change one has,
\begin{align*}
	g(\mathbf{H}_\Sigma^h, \mathbf{T}^{\hat{N}}) &= g(\mathbf{H}_\Sigma^g, \mathbf{T})- \frac{g(\mathbf{A}_\Sigma^g(\mathbf{T}^\top, \mathbf{T}^\top), \mathbf{T})}{1+|\mathbf{T}^\top|_g^2}+\frac{\frac{4g_{\Real}(\mathbf{T}^\top, \mathbf{X}) }{1-s^2} |\mathbf{T}|_g^2-2g(J_{\Real}(\mathbf{T}^\top), \mathbf{T})}{1+|\mathbf{T}|_g^2}\\
	&=g_{\Real}(\mathbf{H}_\Sigma^\Real, \mathbf{T})-\frac{2m g_{\Real}(\mathbf{X}^\perp, \mathbf{T})}{1-s^2}-\frac{g (\mathbf{A}_\Sigma^\Real(\mathbf{T}^\top, \mathbf{T}^\top)-\frac{1}{2} |\mathbf{T}^\top|_g^2 \mathbf{X}^\perp, \mathbf{T})}{ 1+|\mathbf{T}^\top|_g^2}\\
	&+\frac{2g(\mathbf{T}^\top, \mathbf{X})}{1+|\mathbf{T}|_g^2}+\frac{4 g_{\Real}(\mathbf{T}^\top, \mathbf{X}) |\mathbf{T}|_g^2}{(1-s^2)(1+|\mathbf{T}|_g^2)}\\
	&=g_{\Real}(\mathbf{H}_\Sigma^\Real, \mathbf{T})+\frac{2m g_{\Real}(\mathbf{T}^\top, \mathbf{X})}{1-s^2}-\frac{g (\mathbf{A}_\Sigma^\Real(\mathbf{T}^\top, \mathbf{T}^\top), \mathbf{T})}{ 1+|\mathbf{T}^\top|_g^2}\\
	&-\frac{2|\mathbf{T}^\top|^2_{g}g_\Real(\mathbf{T}^\top, \mathbf{X})}{(1-s^2)(1+|\mathbf{T}^\top|_g^2)} +\frac{2g(\mathbf{T}^\top, \mathbf{X})}{1+|\mathbf{T}|_g^2}+\frac{4 g_{\Real}(\mathbf{T}^\top, \mathbf{X}) |\mathbf{T}|_g^2}{(1-s^2)(1+|\mathbf{T}|_g^2)}.
\end{align*}
Hence,
\begin{align*}
g(\mathbf{H}_\Sigma^h, \mathbf{T}^{\hat{N}})&=	g_{\Real}(\mathbf{H}_\Sigma^\Real, \mathbf{T})-\frac{g (\mathbf{A}_\Sigma^\Real(\mathbf{T}^\top, \mathbf{T}^\top), \mathbf{T})}{ 1+|\mathbf{T}^\top|_g^2}+\frac{2g_\Real(\mathbf{T}^\top, \mathbf{X})}{(1-s^2)(1+|\mathbf{T}^\top|_g^2)} \\
&+\frac{g_\Real(\mathbf{T}^\top, \mathbf{X})}{(1+s^2)^2}\left(\frac{8}{(1-s^2)^2}-\frac{4}{1-s^2}\right)+\frac{2(m+1) g_{\Real}(\mathbf{T}^\top, \mathbf{X})}{1-s^2}\\
&=	g_{\Real}(\mathbf{H}_\Sigma^\Real, \mathbf{T})-\frac{g (\mathbf{A}_\Sigma^\Real(\mathbf{T}^\top, \mathbf{T}^\top), \mathbf{T})}{ 1+|\mathbf{T}^\top|_g^2}+\frac{2g_\Real(\mathbf{T}^\top, \mathbf{X})}{(1-s^2)(1+|\mathbf{T}^\top|_g^2)} \\
&+\frac{4g_\Real(\mathbf{T}^\top, \mathbf{X})}{(1+s^2)(1-s^2)^2}+\frac{2(m+1) g_{\Real}(\mathbf{T}^\top, \mathbf{X})}{1-s^2}.
\end{align*}
The second formula follows.
\end{proof}
\begin{cor} \label{MinimalConeLem} If $\Gamma\subset \mathbb{S}^{2n+1}$ is  minimal in $g_{\mathbb{S}}$ and horizontal and $\Sigma$ is the (Euclidean) cone over $\Gamma$ with vertex $\mathbf{0}$ restricted to $\mathbb{B}^{2n+2}$, then ${\Sigma}\setminus \set{\mathbf{0}}$ is minimal in $h=g_{\tilde{B}}$ and isotropic with respect to $\omega_{\tilde{B}}$. 
\end{cor}
\begin{proof}
As $\Gamma$ is minimal in $g_{\mathbb{S}}$,  $\mathbf{H}^\Real_\Sigma=\mathbf{0}$ on $\Sigma\setminus \set{\mathbf{0}}$.  Likewise, as $\Gamma$ is horizontal, $\mathbf{T}^\top=\mathbf{0}$ on $\Sigma\setminus \set{\mathbf{0}}$ and so $\mathbf{H}_\Sigma^h=\mathbf{0}$.  This also implies that ${\Sigma}\setminus \set{\mathbf{0}}$ is isotropic.  
\end{proof}
\subsection{Asymptotic behavior of minimal submanifolds in $\mathbb{CH}^{n+1}$}\label{MinCompactSec}
We use Proposition \ref{2ndFFMCProp} to show an improvement of boundary regularity for minimal submanifolds of the modified Bergman metric.  That is, we prove Theorem \ref{main1}. 
%

We first show a preliminary partial result.
\begin{lem}\label{BoundaryTtanLem}
	Let $\Sigma\subset\bar{\mathbb{B}}^{2n+2}$ be an $m$-dimensional $C^2$-regular submanifold with boundary that meets $\partial \mathbb{B}^{2n+2}$ transversely along $\partial \Sigma\subset \partial \mathbb{B}^{2n+2}$.  If $\mathbf{T}^\top=\mathbf{0}$ on $\partial \Sigma$, then
	$$
	\mathbf{T}^\top=(1-s) \frac{1}{|\mathbf{X}^\top|^2} \left((J_{\Real}(\mathbf{X}^\top))^\top-S_{\mathbf{T}^\perp}^{\tilde{\Sigma}}(\mathbf{X}^\top)\right).
	$$
	Moreover, along $\partial \Sigma$, we may write 
	$$
	\mathbf{X}^\top=|\mathbf{X}^\top|^2_\Real\mathbf{X}+Z=|\mathbf{X}^\top|^2_\Real\mathbf{X}+J_{\mathbb{S}}(Z_1)+Z_2,
	$$
	where $Z$ is normal to $\mathbf{X}$, $\partial {\Sigma}$ and ${\mathbf{T}}$,  $Z_1$ is tangent to $\partial \Sigma$, $Z_2$ and $J_{\mathbb{S}}(X_2)$ are normal to $\partial \tilde{\Sigma}$, $\mathbf{X}$, and $\mathbf{T}$. 
	Using this decomposition we obtain, along $\partial \Sigma$,
	\begin{align*}
		(J_{\Real}(\mathbf{X}^\top))^\top=-Z_1\mbox{ and }	\mathbf{S}_{\mathbf{T}^\perp}^{{\Sigma}}(\mathbf{X}^\top)=g_{\Real}(\mathbf{T}, \mathbf{A}_{{\Sigma}}^\Real (\mathbf{X}^\top, \mathbf{X}^\top))\frac{\mathbf{X}^\top}{|\mathbf{X}^\top|_{\Real}^2}+Z_1 .
	\end{align*}
\end{lem}
\begin{proof}
	The hypotheses on $\Sigma$ ensure that $\partial \Sigma$ is horizontal and we can write
	$$
	\mathbf{T}^\top=(1-s)\mathbf{v}
	$$
	for a $C^2$ vector field $\mathbf{v}$ tangent to $\Sigma$ that extends in a $C^1$ fashion to the boundary.
	We compute that on the boundary 
	$$
	\nabla_{\mathbf{X}^\top}^\Sigma \mathbf{T}^\top =(\nabla^\Real_{\mathbf{X}^\top}( \mathbf{T}-\mathbf{T}^\perp))^\top= -(J_{\Real}(\mathbf{X}^\top))^\top+S_{\mathbf{T}^\perp}^\Sigma(\mathbf{X}^\top).
	$$
	As $\mathbf{X}^\top \cdot s= |\mathbf{X}^\top|^2$ on the boundary we verify that
	$$
	\mathbf{v}= -\frac{1}{|\mathbf{X}^\top|^2} \left(-(J_{\Real}(\mathbf{X}^\top))^\top+S_{\mathbf{T}^\perp}^\Sigma(\mathbf{X}^\top)\right).
	$$
	
	On $\partial \Sigma$ we may write
	$$
	\mathbf{X}^\top =|\mathbf{X}^\top|^2 \mathbf{X}+ {Z}=|\mathbf{X}^\top|^2 \mathbf{X}+J_{\mathbb{S}}(Z_1)+Z_2
	$$
	where $Z$ is tangent to $\mathbb{S}^{2n+1}=\partial \mathbb{B}^{2n+2}$ and thus orthogonal to $\partial \Sigma$.  The hypotheses on $\Sigma$ ensure that $\mathbf{X}^\top$, and, thus ${Z}$, are orthogonal to $\mathbf{T}$ and $\partial \Sigma$.  Hence,  we may further decompose $Z$ so that $Z_1$ is tangent to $\partial \Sigma$ and both $Z_2$ and $J_{\mathbb{S}}(Z_2)$ are orthogonal to $\partial \Sigma$.  If $\partial \Sigma$ is Legendrian, then $Z_2$ is zero.
	Clearly, along $\partial \Sigma$,
	$$
	(J_{\Real}(\mathbf{X}^\top))^\top= (-|\mathbf{X}^\top|^2 \mathbf{T}+ J_{\mathbb{S}}(Z))^\top=-Z_1.
	$$

	
For $Y$ tangent to $\partial \Sigma$, the hypotheses on $\Sigma$ ensure
	\begin{align*}
		g_{\Real}(\mathbf{T}, \mathbf{A}_\Sigma^\Real(\mathbf{X}^\top, {Y}))&= g_{\Real}(\mathbf{T}, \nabla_{{Y}}^\Real \mathbf{X}^\top)=g_{\Real}(\mathbf{T}, \nabla_{{Y}}^\Real \left( |\mathbf{X}^\top|_\Real^2 \mathbf{X}+ Z\right) )\\
		&= g_{\Real}(\mathbf{T}, \nabla_{{Y}}^\Real  Z )= {Y} \cdot g_\Real(\mathbf{T}, Z) - g_{\Real}(\nabla^\Real_Y \mathbf{T}, Z)\\
		&= g_{\Real}(J_{\Real}(Y), Z)=-g_{\Real}(J_{\Real}(Z), Y)\\
		&=-g_{\Real}((J_{\mathbb{S}}(Z))^\top, Y)=g_{\Real}(Z_1, Y).
	\end{align*}
	Hence, along $\partial \Sigma$, we have
	$$
	\mathbf{S}_{\mathbf{T}^\perp}^\Sigma(\mathbf{X}^\top)=\mathbf{S}_{\mathbf{T}}^\Sigma(\mathbf{X}^\top)=g_{\Real}(\mathbf{T}, \mathbf{A}_\Sigma^\Real (\mathbf{X}^\top, \mathbf{X}^\top))\frac{\mathbf{X}^\top}{|\mathbf{X}^\top|_{\Real}^2}+Z_1.
	$$
\end{proof}

We now prove Theorem \ref{main1}
\begin{proof}[Proof of Theorem \ref{main1}]
Let us choose a modified Bergman compactification, $\tilde{\Sigma}$, of $\Sigma$.  The hypotheses ensure that $\tilde{\Sigma}$ is $C^2$ up to $\partial \mathbb{B}^{2n+2}$, meets the boundary transversally, has $\partial \tilde{\Sigma}$ horizontal, and  is minimal with respect to ${g}_{\tilde{B}}$.  For simplicity, we will write $\Sigma$ instead of $\tilde{\Sigma}$ for the remainder of the proof. 

Our hypotheses ensure that near the boundary
\begin{equation}\label{TtopEqn}
\mathbf{T}^\top=g_{\Real}(\mathbf{T}^\top, \mathbf{X}) \frac{\mathbf{X}^\top}{|\mathbf{X}^\top|_\Real^2}+ \mathbf{V}
\end{equation}
where $\mathbf{V}$ is tangent to $\Sigma$, orthogonal to $\mathbf{X}^\top$, and vanishes along the boundary.   This is well defined as the hypotheses ensure that $\mathbf{X}^\top\neq 0$ along the boundary.

Using the minimality of $\Sigma$ along with the fact that $\Sigma$ is $C^2$ up to the boundary, the second formula of Proposition \ref{2ndFFMCProp} can be rewritten as
\begin{align*}
	0&= (	(1-s^2)^2+4|\mathbf{T}^\top|_\Real^2)\left(\frac{1-s^2}{2} g_{\Real} (\mathbf{H}_\Sigma^\Real, \mathbf{T})+\left( m+1+\frac{2}{1+s^2}\right) g_{\Real}(\mathbf{T}^\top, \mathbf{X})\right)\\
	&+2(1-s^2)\left( \frac{1-s^2}{2} g_\Real(\mathbf{T}^\top, \mathbf{X}) - g_\Real(\mathbf{A}_\Sigma^\Real(\mathbf{T}^\top, \mathbf{T}^\top), \mathbf{T})\right). 
	\end{align*}
From this expression one obtains that, near $\partial \Sigma$,
\begin{align*}
4(m+2)|\mathbf{T}^\top|_\Real^2 g_{\Real}(\mathbf{T}^\top, \mathbf{X})&=O(1-s).
\end{align*}
Hence, on $\partial \Sigma$, either $\mathbf{T}^\top=\mathbf{0}$ or $g_{\Real}(\mathbf{T}^\top, \mathbf{X})=0$.  In the latter case, \eqref{TtopEqn}  implies it is still true that $\mathbf{T}^\top=\mathbf{0}$ on $\partial \Sigma$ and, thus, there is a vector field, $\mathbf{v}$,  tangent to $\Sigma$ such that
$$
\mathbf{T}^\top=(1-s) \mathbf{v}.
$$

From the first formula of Proposition \ref{2ndFFMCProp}, it follows that near $\partial \Sigma$, one has
\begin{align*}
	\mathbf{0}&= \frac{2}{1+ |\mathbf{v}|_\Real^2} (J_{\Real}(\mathbf{v}))^{\tilde{N}}+ \left( m+1 - \frac{1}{1+|\mathbf{v}|_\Real^2}\right) \mathbf{X}^{\tilde{N}}+o(1).
\end{align*}
Hence, along the boundary 
$$
\mathbf{0}=2 (J_{\Real}(\mathbf{v}))^{\tilde{N}}+\left( (m+1)|\mathbf{v}|_\Real^2+m\right) \mathbf{X}^{\tilde{N}}.
$$
As we have shown $\mathbf{T}^{\perp}=\mathbf{T}$ on $\partial \Sigma$, we may appeal to Lemma \ref{BoundaryTtanLem} to see
$$
\mathbf{X}^{\tilde{N}}=\mathbf{X}^\perp=\mathbf{X}-\mathbf{X}^\top=|\mathbf{X}^\perp|_\Real^2 \mathbf{X}- Z= |\mathbf{X}^\perp|_\Real^2 \mathbf{X}-J_{\mathbb{S}}(Z_1)-Z_2
$$
where $Z, Z_1$ and $Z_2$ are as in the statement of Lemma \ref{BoundaryTtanLem}.
By Lemma \ref{BoundaryTtanLem},  for appropriate $\beta$, on $\partial \Sigma$ one has
$$
\mathbf{v}= -\frac{2}{|\mathbf{X}^\top|^2}  J_{\mathbb{S}}(Z_1) +\beta\mathbf{X}=-\frac{2}{|\mathbf{X}^\top|^2} Z_1 +\beta (|\mathbf{X}^\top|^2 \mathbf{X}+J_{\mathbb{S}}(Z_1)+Z_2).
$$
Hence, 
\begin{align*}
(J_{\Real}(\mathbf{v}))^{\tilde{N}}&=\left(-\frac{2}{|\mathbf{X}^\top|^2}  J_{\mathbb{S}}(Z_1) +\beta (|\mathbf{X}^\top|^2 \mathbf{T}-Z_1+J_{\mathbb{S}}(Z_2))\right)^{\tilde{N}} \\
&= -\frac{2}{|\mathbf{X}^\top|^2}  J_{\mathbb{S}}(Z_1)^\perp +\beta (J_{\mathbb{S}}(Z_2))^\perp.
\end{align*}
Plugging this into the previous identity we obtain
$$
\mathbf{0}= -\frac{4}{|\mathbf{X}^\top|^2}  J_{\mathbb{S}}(Z_1)^\perp +2\beta (J_{\mathbb{S}}(Z_2))^\perp+\left( (m+1)|\mathbf{v}|^2+m\right) \mathbf{X}^\perp.
$$
As $J_{\mathbb{S}}(Z_2)$ is, by construction, orthogonal to $\partial \Sigma$, $\mathbf{X}^\perp$,  $\mathbf{X}^\top$, and $\mathbf{T}$, one has
$$
J_{\mathbb{S}}(Z_2)^{\tilde{N}}= J_{\mathbb{S}}(Z_2)^\perp = J_{\mathbb{S}}(Z_2).
$$
Moreover, as $J_{\mathbb{S}}(Z_1)$ is orthogonal to $\partial \Sigma$, 
$$
(J_{\mathbb{S}}(Z_1))^\perp  =J_{\mathbb{S}}(Z_1)-g_{\Real}(J_{\mathbb{S}}(Z_1), \mathbf{X}^\top) \frac{\mathbf{X}^\top}{|\mathbf{X}^\top|^2},
$$
which immediately implies that $J_{\mathbb{S}}(Z_2)$ is also orthogonal to $(J_{\mathbb{S}}(Z_1))^\perp$.
Hence,
$$
\beta J_{\mathbb{S}}(Z_2)=\mathbf{0}=-\frac{4}{|\mathbf{X}^\top|^2}  J_{\mathbb{S}}(Z_1)^\perp + \left( (m+1)|\mathbf{v}|^2+m\right) \mathbf{X}^\perp.
$$
This implies $\beta J_{\mathbb{S}}(Z_2)=0$ and 
\begin{align*}
0&=-\frac{4}{|\mathbf{X}^\top|^2}  |J_{\mathbb{S}}(Z_1)^\perp|^2 + \left( (m+1)|\mathbf{v}|^2+m\right) g_{\Real}(\mathbf{X}^\perp, J_{\mathbb{S}}(Z_1))\\
& = -\frac{4}{|\mathbf{X}^\top|^2}  |J_{\mathbb{S}}(Z_1)^\perp|^2 -\left( (m+1)|\mathbf{v}|^2+m\right)|J_{\mathbb{S}}(Z_1)|^2.
\end{align*}
It follows that $Z_1=J_{\mathbb{S}}(Z_1)=(J_{\mathbb{S}}(Z_1))^\perp=\mathbf{0}$.  Hence,
$$
\mathbf{0}=\left( (m+1)|\mathbf{v}|^2+m\right) \mathbf{X}^\perp
$$
and so $\mathbf{X}^\perp=0$.  In particular, as a submanifold of $\mathbb{CH}^{n+1}$, $\Sigma$ is asymptotically quasi-normal. 
Moreover, along $\partial \Sigma$, one has a function $\alpha$ so
$$
\mathbf{v}=-\frac{1}{|\mathbf{X}^\top|^2}S_{\mathbf{T}^\perp}^\Sigma(\mathbf{X}^\top)=-\mathbf{S}_{\mathbf{T}}^\Sigma(\mathbf{X})=g_{\Real}(\mathbf{T}, \mathbf{A}_\Sigma^\Real(\mathbf{X}^\top, \mathbf{X}^\top)) \mathbf{X}^\top=\alpha \mathbf{X}^\top.
$$

As $\Sigma$ meets $\partial \mathbb{B}^{2n+2}$ orthogonally, for $X,Y$ tangent to $\partial \Sigma$, 
$$
 \mathbf{A}_{\partial \Sigma}^\Real(X,Y)=(\nabla_X^\Real Y)^\perp=\mathbf{A}_\Sigma^\Real(X,Y).
 $$
Moreover, as $\mathbf{T}^\top=0$,  for such $X,Y$,
\begin{align*}
g_{\Real}(\mathbf{T},\mathbf{A}_\Sigma^\Real(X,Y) )&=g_{\Real}(\mathbf{T}, \nabla^\Real_{X} Y)=X\cdot g_{\Real}(\mathbf{T}, Y)- g_{\Real}(\nabla_{X}^\Real \mathbf{T}, Y)\\
&=g_{\Real}(J_{\Real}(X), Y)=0.
\end{align*}
Hence, along the boundary, 
$$
\mathbf{H}_\Sigma^\Real= \mathbf{H}_{\partial \Sigma}^{\mathbb{S}}+\mathbf{A}_\Sigma^\Real(\mathbf{X}^\top, \mathbf{X}^\top)\mbox{ and } g_{\Real}(\mathbf{H}^{\Real}_{\partial \Sigma}, \mathbf{T})=g_{\mathbb{S}}(\mathbf{H}^{\mathbb{S}}_{\partial \Sigma}, \hat{\mathbf{T}})=0.
$$
It follows that,
$$g_{\Real}(\mathbf{H}_\Sigma^\Real, \mathbf{T})= g_{\Real}(\mathbf{A}_\Sigma^\Real(\mathbf{X}^\top, \mathbf{X}^\top), \mathbf{T})=g_{\Real}(\mathbf{v}, \mathbf{X})=\alpha.
$$
Using $\mathbf{T}^\top=(1-s)\mathbf{v}$, the second formula of Proposition \ref{2ndFFMCProp} reduces, on  $\partial \Sigma$, to:
\begin{align*}
0&=(1+|\mathbf{v}|^2_\Real)( g_{\Real}(\mathbf{H}_\Sigma^\Real, \mathbf{T})+(m+2) g_{\Real}(\mathbf{v}, \mathbf{X}))+ g_{\Real}(\mathbf{v}, \mathbf{X}) -g_{\Real}(\mathbf{A}_\Sigma^\Real(\mathbf{v}, \mathbf{v}), \mathbf{T})\\
&=(1+\alpha^2)( \alpha+(m+2)\alpha))+\alpha-\alpha^3=(m+2) \alpha^3+(m+4) \alpha.
\end{align*}
This is readily seen to have a solution only when $\alpha=0$.  Hence,  $\mathbf{v}=\mathbf{0}$ and $g_{\Real}(\mathbf{A}_\Sigma^\Real(\mathbf{X}^\top, \mathbf{X}^\top), \mathbf{T})=0$ along $\partial \Sigma$.  It follows from Lemma \ref{WeakAsy2StrongAsyLem} that $\Sigma$, thought of as a submanifold of $\mathbb{CH}^{n+1}$, is $C^1$-asymptotically regular and strongly asymptotically horizontal.
\end{proof}

While we do not use it elsewhere in this paper, we record some finer information about the boundary geometry of a compactified minimal surface under additional regularity hypotheses.
\begin{cor}
		Suppose that $\Sigma\subset \mathbb{CH}^{n+1}$ is an $m$-dimensional  minimal submanifold.  If $\Sigma$ is weakly $C^3$-asymptotically regular and weakly asymptotically horizontal, and $\tilde{\Sigma}\subset \bar{\mathbb{B}}^{2n+2}$ is a modified Bergman compactification of $\Sigma$, then, on $\partial \tilde{\Sigma}$,
		\begin{enumerate}
			\item $g_{\Real}(\mathbf{A}_{\tilde{\Sigma}}^\Real(X,Y), \mathbf{T})=0$ for any tangent vectors $X,Y$.
			\item $\mathbf{A}_{\tilde{\Sigma}}^\Real (\mathbf{X}^\top, Y)= \frac{1}{m+1} \mathbf{H}^\mathbb{S}_{\partial \tilde{\Sigma}} g_{\Real}(\mathbf{X}^\top, Y)$ for any tangent vector $Y$.
	\end{enumerate}
\end{cor}
\begin{proof}
Our hypotheses ensure that $\tilde{\Sigma}$ is $C^3$ up to $\partial \mathbb{B}^{2n+2}$, meets the boundary transversally, has $\partial \tilde{\Sigma}$ horizontal and  is minimal with respect to $g_{\tilde{B}}$.  For simplicity, we will write $\Sigma$ instead of $\tilde{\Sigma}$ for the remainder of the proof.  Observe, that by Theorem \ref{main1}, there is a vector field, $\mathbf{w}$, tangent to $\Sigma$ such that 
$$
\mathbf{T}^\top=(1-s)^2 \mathbf{w}.
$$
This also establishes (1) as in Lemma \ref{WeakAsy2StrongAsyLem}.
	
The first formula of Proposition \ref{2ndFFMCProp} implies that near $\partial \Sigma$, 
	\begin{align*}
		\mathbf{0}&=  (1-s)(\mathbf{H}_{ \Sigma}^\Real)^{\tilde{N}} -2 (1-s) (J_{\Real}(\mathbf{w}))^{\tilde{N}}-m \mathbf{X}^{\tilde{N}}+o((1-s)).
	\end{align*}
	To go further it is helpful to observe that
	\begin{align*}
		\mathbf{X}^{\tilde{N}}&= \mathbf{X}^\perp-g_{\Real}(\mathbf{X}, \mathbf{T}^\perp) \frac{\mathbf{T}^\perp}{|\mathbf{T}^\perp|_{\Real}^2}=\mathbf{X}^\perp+g_{\Real}(\mathbf{X}, \mathbf{T}^\top) \frac{\mathbf{T}^\perp}{|\mathbf{T}^\perp|_{\Real}^2}\\
		&= \mathbf{X}^\perp+ O((1-s)^2).
	\end{align*}
	As $(\nabla_{\mathbf{X}^\top}^\Real\mathbf{X}^\perp)^\perp=-\mathbf{A}_{\Sigma}^\Real(\mathbf{X}^\top, \mathbf{X}^\top)$,  it follows that near the boundary
	$$
	\mathbf{X}^{\tilde{N}}= (1-s)(\mathbf{A}_{\Sigma}^{\Real}(\mathbf{X}^\top, \mathbf{X}^\top))^{\tilde{N}}+o(1-s).
	$$
	Hence, we conclude that
	$$
	\mathbf{0}= \left( \mathbf{H}_\Sigma^\Real-2J_{\Real}(\mathbf{w})-m \mathbf{A}^\Real_\Sigma(\mathbf{X}^\top, \mathbf{X}^\top)\right)^{\tilde{N}}.
	$$
	From these computations and item (1), this means that, on $\partial \Sigma$,
	$$
	2(J_{\Real}(\mathbf{w}))^{\tilde{N}}=\left((1-m) \mathbf{A}_\Sigma^\Real(\mathbf{X}^\top, \mathbf{X}^\top)+\mathbf{H}_{\partial \Sigma}^{\mathbb{S}}\right)^{\tilde{N}}=(1-m) \mathbf{A}_\Sigma^\Real(\mathbf{X}^\top, \mathbf{X}^\top)+\mathbf{H}_{\partial \Sigma}^{\mathbb{S}}.
	$$
	
	To complete the proof we need to compute $\mathbf{w}$ in terms of the geometry of $\Sigma$. To that end, observe that
	$$
	\nabla_{\mathbf{X}^\top}^\Real \mathbf{X}^\top= \mathbf{X}^\top+\mathbf{S}_{\mathbf{X}^\perp}^\Sigma(\mathbf{X}^\top)+\mathbf{A}_\Sigma^\Real(\mathbf{X}^\top, \mathbf{X}^\top).
	$$ 
	Hence,
	\begin{align*}
		\nabla^\Sigma_{\mathbf{X}^\top} (J_{\Real}(\mathbf{X}^\top))^\top&= \left( \nabla_{\mathbf{X}^\top}^\Real \left(J_{\Real}(\mathbf{X}^\top) - (J_{\Real}(\mathbf{X}^\top))^\perp\right) \right)^\top\\
		&=(J_{\Real}(\mathbf{X}^\top+\mathbf{S}_{\mathbf{X}^\perp}^\Sigma(\mathbf{X}^\top)+\mathbf{A}_\Sigma^\Real(\mathbf{X}^\top, \mathbf{X}^\top)))^\top+\mathbf{S}_{(J_{\Real}(\mathbf{X}^\top))^\perp} ^\Sigma(\mathbf{X}^\top).
	\end{align*}
	The properties already established about the boundary geometry of $\Sigma$ yield,
	$$
	\nabla^\Sigma_{\mathbf{X}^\top} (J_{\Real}(\mathbf{X}^\top))^\top=J_{\Real}(\mathbf{A}_\Sigma^\Real(\mathbf{X}^\top, \mathbf{X}^\top))-\mathbf{S}_{\mathbf{T}^\perp}^\Sigma(\mathbf{X}^\top)=J_{\Real}(\mathbf{A}_\Sigma^\Real(\mathbf{X}^\top, \mathbf{X}^\top)).
	$$
	We further compute that
	\begin{align*}
		g_\Real(&\nabla_{\mathbf{X}^\top} (\mathbf{S}_{\mathbf{T}^\perp}^\Sigma(\mathbf{X}^\top)), Y) = \mathbf{X}^\top \cdot g((\mathbf{S}_{\mathbf{T}^\perp}^\Sigma(\mathbf{X}^\top), Y )- g((\mathbf{S}_{\mathbf{T}^\perp}^\Sigma(\mathbf{X}^\top), \nabla_{\mathbf{X}^\top} Y)\\
		&=\mathbf{X}^\top \cdot g_{\Real}( \mathbf{T}^\perp, \mathbf{A}_{\Sigma}^\Real(\mathbf{X}^\top, Y))- g((\mathbf{S}_{\mathbf{T}^\perp}^\Sigma(\mathbf{X}^\top), \nabla_{\mathbf{X}^\top} Y)\\
		&= g_{\Real} (-(J_{\Real}(\mathbf{X}^\top))^\perp-\mathbf{A}_{\Sigma}^\Real(\mathbf{X}^\top, \mathbf{T}^\top), \mathbf{A}_{\Sigma}^\Real(\mathbf{X}^\top, Y))\\
		&+g_{\Real} (\mathbf{T}^\perp, (\nabla_{\mathbf{X}^\top}^\perp \mathbf{A}_\Sigma^\Real)(\mathbf{X}^\top, Y) +g_{\Real}(\mathbf{T}^\perp, \mathbf{A}_\Sigma^\Real( \nabla_{\mathbf{X}^\top}^\Real \mathbf{X}^\top, Y))\\
		&=  g_{\Real} (-(J_{\Real}(\mathbf{X}^\top))^\perp-\mathbf{A}_{\Sigma}^\Real(\mathbf{X}^\top, \mathbf{T}^\top), \mathbf{A}_{\Sigma}^\Real(\mathbf{X}^\top, Y))+g_{\Real} (\mathbf{T}^\perp, (\nabla_{Y}^\perp \mathbf{A}_\Sigma^\Real)(\mathbf{X}^\top, \mathbf{X}^\top) )\\
		&+g_{\Real}(\mathbf{T}^\perp, \mathbf{A}_\Sigma^\Real(  \mathbf{X}^\top+\mathbf{S}_{\mathbf{X}^\perp}^\Sigma(\mathbf{X}^\top), Y))\\
		&=g_{\Real} (-(J_{\Real}(\mathbf{X}^\top))^\perp-\mathbf{A}_{\Sigma}^\Real(\mathbf{X}^\top, \mathbf{T}^\top), \mathbf{A}_{\Sigma}^\Real(\mathbf{X}^\top, Y))\\
		&+ g_{\Real}(\mathbf{T}^\perp, \mathbf{A}_\Sigma^\Real(  \mathbf{X}^\top+\mathbf{S}_{\mathbf{X}^\perp}^\Sigma(\mathbf{X}^\top), Y))\\
		&+Y\cdot g_\Real(\mathbf{T}^\perp, \mathbf{A}_{\Sigma}^\Real(\mathbf{X}^\top, \mathbf{X}^\top)) + g_\Real((J_{\Real}(Y))^\perp+\mathbf{A}_{\Sigma}^\Real(Y, \mathbf{T}^\top), \mathbf{A}_{\Sigma}^\Real(\mathbf{X}^\top, \mathbf{X}^\top)).
	\end{align*}
	On the boundary, as $\mathbf{T}^\top=\mathbf{0}$ and as item (1) holds, this simplifies to
	\begin{align*}
		g_{\Real}(\nabla_{\mathbf{X}^\top}^\Real &(\mathbf{S}_{\mathbf{T}^\perp}^\Sigma(\mathbf{X}^\top)), Y) =-g_{\Real}(J_{\Real}(\mathbf{A}_{\Sigma}^\Real(\mathbf{X}^\top, \mathbf{X}^\top), Y)+Y \cdot  g_{\Real}(\mathbf{T}^\perp, \mathbf{A}_{\Sigma}^\Real(\mathbf{X}^\top, \mathbf{X}^\top)) \\
		&= -g_{\Real}(J_{\Real}(\mathbf{A}_{\Sigma}^\Real(\mathbf{X}^\top, \mathbf{X}^\top), Y)+\mathbf{X}^\top \cdot  g_{\Real}(\mathbf{T}^\perp, \mathbf{A}_{\Sigma}^\Real(\mathbf{X}^\top, \mathbf{X}^\top)) g_{\Real}(Y, \mathbf{X}) \mathbf{X},
	\end{align*}
	where we used $\mathbf{X}^\top=\mathbf{X}$ on $\partial \Sigma$. 
	Hence, on the boundary,
	\begin{align*}
		\mathbf{w}&= \frac{1}{2} \nabla_{\mathbf{X}^\top}^\Sigma\nabla_{\mathbf{X}^\top}^\Sigma \mathbf{T}^\top=-J_{\Real}(\mathbf{A}_{\Sigma}^\Real(\mathbf{X}^\top, \mathbf{X}^\top))+\frac{1}{2}\mathbf{X}^\top \cdot  g_{\Real}(\mathbf{T}^\perp, \mathbf{A}_{\Sigma}^\Real(\mathbf{X}^\top, \mathbf{X}^\top))  \mathbf{X}
	\end{align*}
and so
	\begin{align*}
		(J_{\Real}(\mathbf{w}))^{\tilde{N}}= \left( \mathbf{A}_{\Sigma}^\Real(\mathbf{X}^\top, \mathbf{X}^\top)-\frac{1}{2}\mathbf{X}^\top \cdot  g_{\Real}(\mathbf{T}^\perp, \mathbf{A}_{\Sigma}^\Real(\mathbf{X}^\top, \mathbf{X}^\top)) \mathbf{T}\right)^{\tilde{N}}=\mathbf{A}_{\Sigma}^\Real(\mathbf{X}^\top, \mathbf{X}^\top)
	\end{align*}
where we used item (1).  Hence, 
	$$
	2\mathbf{A}_{\Sigma}^{\Real}(\mathbf{X}^\top, \mathbf{X}^\top)=(1-m) \mathbf{A}_\Sigma^\Real(\mathbf{X}^\top, \mathbf{X}^\top)+\mathbf{H}_{\partial \Sigma}^{\mathbb{S}},
	$$
	from which item (2) with $Y=\mathbf{X}^\top$ follows.  Item (2) for $Y$ orthogonal to $\mathbf{X}^\top$ is straightforward.
\end{proof}

\section{Colding-Minicozzi entropy in $\mathbb{CH}^{n+1}$}
\label{EntropySec}
In this section, we prove Theorem \ref{main2}.  To do so we first introduce some notation.  Let ${\Upsilon}, {\Upsilon}':\mathbb{CH}^{n+1}\to \mathbb{B}^{2n+2}$ be two choices of Bergman compactifications with the property that ${\Upsilon}(p_0)={\Upsilon}'(p_0')=\mathbf{0}$. By construction, there is an element $\Phi\in Aut_{g_{{B}}}^+(\mathbb{B}^{2n+2})$ such that ${\Upsilon}'=\Phi\circ {\Upsilon}$.  Moreover, when $p_0=p_0'$, this element $\Phi$ is induced by an element of $\mathbf{U}(n+1)$ -- i.e., it is an isometry of $g_{\Real}$. As a consequence,  such compactifications endow $\partial_{\infty} \mathbb{CH}^{n+1}$ with a well-defined Riemannian metric, which is obtained from the standard metric $g_{\mathbb{S}}$ on $\mathbb{S}^{2n+1}=\partial \mathbb{B}^{2n+2}$.  While this metric depends on $p_0$, it is otherwise independent of the choices in the compactification. Let us denote this metric by $g_{\partial_\infty \mathbb{CH}}^{p_0}$. 
Clearly, $g_{\partial_\infty \mathbb{CH}}^{p_0}$ and $g_{\partial_\infty \mathbb{CH}}^{q_0}$ are related by the action of $Aut_{CR}(\mathbb{S}^{2n+1})$ for different choices of distinguished points $p_0$ and $q_0$. Therefore any geometric quantity defined on $\mathbb{S}^{2n+1}$ that is $ Aut_{CR}(\mathbb{S}^{2n+1})$ invariant is well-defined on $\partial_\infty \mathbb{CH}^{n+1}$.    Furthermore, as $\bar{\mathcal{S}}$ is an identity on $\mathbb{S}^{2n+1}$ one could also use the modified Bergman compactification to the same effect.

Fix an $m$-dimensional $C^l$ submanifold  $\Gamma\subset \partial_\infty \mathbb{CH}^n$ and let ${\Upsilon}(\Gamma)\subset \mathbb{S}^{2n+1}$ be the corresponding submanifold of the sphere associated to the Bergman compactification, $\Upsilon$.  Set
$$
Vol_{\partial_\infty \mathbb{CH}}(\Gamma,p_0)=|{\Upsilon}(\Gamma))|_{\mathbb{S}}.
$$
When $p_0'$ is a different choice of a distinguished point, then, by the above discussion there is an element, $\Psi\in \mathrm{Aut}_{CR}(\mathbb{S}^{2n+1})$ such that
$$
Vol_{\partial_\infty \mathbb{CH}}(\Gamma,p_0')=|\Psi'(\Gamma)|_{\mathbb{S}}=|\Psi({\Upsilon}(\Gamma))|_{\mathbb{S}}.
$$
Hence, the \emph{CR-volume}  of a horizontal submanifold, $\Gamma\subset \partial_\infty \mathbb{CH}^{n+1}$, defined by
$$
\lambda_{CR}[\Gamma]=\sup_{\Psi\in \mathrm{Aut}_{CR}(\mathbb{S}^{2n+1})} |\Psi({\Upsilon}(\Gamma))|_{\mathbb{S}}=\sup_{p_0\in \mathbb{CH}^{n+1}} Vol_{\partial_\infty \mathbb{CH}^n}(\Gamma,p_0)
$$
is a well-defined quantity independent of choices.  One will also obtain the same value using modified Bergman compactifications. 

We now begin the proof of Theorem \ref{main2} and first record a basic relationship between geodesic balls of the modified Bergman metric and the Euclidean metric.
\begin{lem}\label{BallLem}
  Setting $s(R)= \tanh\frac{R}{2},$ it follows that
	$${B}_R^{g_{\tilde{B}}}(\mathbf{0})=B_{s(R)}^{\Real}(\mathbf{0}).$$ 
\end{lem}
\begin{proof}
	Suppose that $\ell$ is a line segment in $\mathbb{B}^{2n+2}$ with one endpoint through $\mathbf{0}$.  Proposition \ref{2ndFFMCProp} implies $\ell$ is also a geodesic of $g_{\tilde{B}}$.  The length with respect to $g_{\tilde{B}}$ is 
	$$
	R=	\int_0^s \frac{2}{1-t^2} dt = \ln\left(\frac{1+s}{1-s}\right).
	$$
	Hence,
	$$
	s(R)=\frac{e^R-1}{e^R+1}=\tanh \frac{R}{2}.
	$$
\end{proof}

It is also useful to clarify the relationship between $Vol_{\partial_\infty \mathbb{CH}}(\Gamma; p_0)$ and the geometry of submanifolds asymptotic to $\Gamma$ for varying degrees of asymptotic regularity.
\begin{lem}\label{BoundaryLem}
	Let $\Sigma \subset \mathbb{CH}^{n+1}$ be an $m$-dimensional submanifold that is weakly $C^1$-asymptotically regular and weakly asymptotically horizontal.
	For any $p_0\in \mathbb{CH}^{n+1}$,
	$$
	Vol_{\partial_\infty \mathbb{CH}}(\partial_\infty \Sigma; p_0)\leq \liminf_{r\to \infty} \frac{Vol_{\mathbb{CH}}(\Sigma\cap \partial B_{r}^{\mathbb{CH}}(p_0))}{\sinh^{m-1}(r)}
	$$
 where $B_r^{\mathbb{CH}}(p_0)$ is a ball of radius $r$ centered at $p_0$ in $\mathbb{CH}^{n+1}$.  If $\Sigma$ is $C^1$-asymptotically regular and strongly asymptotically horizontal, then
 	$$
 Vol_{\partial_\infty \mathbb{CH}}(\partial_\infty \Sigma; p_0)= \lim_{r\to \infty} \frac{Vol_{\mathbb{CH}}(\Sigma\cap \partial B_{r}^{\mathbb{CH}}(p_0))}{\sinh^{m-1}(r)}.
 $$
\end{lem}

\begin{proof}
	Let $\tilde{\Upsilon}_{p_0}: \mathbb{CH}^{n+1}\to \mathbb{B}^{2n+2}$ be a modified Bergman compactification sending $p_0$ to $\mathbf{0}$.  By Lemma \ref{BallLem}, one has
	$$
	\tilde{\Upsilon}_{p_0} (\partial B_r^{\mathbb{CH}}(p_0))=\partial B_{r}^{g_{\tilde{B}}}(0)=\partial B_{s}(0)
	$$
	 where $r=\ln \left( \frac{1+s}{1-s}\right) $.
	
	Set $\Sigma_r=\Sigma\cap \partial B_{r}^{\mathbb{CH}}(p_0)$
	and let $\Sigma'=\tilde{\Upsilon}_{p_0}(\Sigma)$.  Clearly, $\tilde{\Upsilon}_{p_0}(\Sigma_r)=\Sigma'\cap \partial B_s(0)=\Sigma'_s$.  Let $g_r$ be the metric on $\Sigma_r$ induced from $g_{\mathbb{CH}}$, $g_{s}'$ be the metric induced on $\Sigma'_s$ from $g_{\Real}$, and $g_s''$ be the metric on $\Sigma'_s$ induced from $g_{\tilde{B}}$.  From the form of $g_{\tilde{B}}$ and properties of $\Sigma'$ near the boundary, that follow from $\Sigma$ being weakly $C^1$-asymptotically regular, one has
\begin{equation}\label{ComparisionMetricEqn}
	g_s''=\frac{4}{(1-s^2)^2} g_s' + \frac{4s^4}{(1-s^2)^4} (i_{\Sigma'_s}^*\theta)^2\geq \frac{4}{(1-s^2)^2} g_s',
\end{equation}
	where the inequality is in the sense of symmetric bilinear forms.
	
	Hence,
	$$
	\tilde{\Upsilon}_{p_0}^*g_{s}'\leq  \left( \frac{e^r-1}{e^r+1}\right)^2\sinh^{-2}(r)	\tilde{\Upsilon}_{p_0}^*g_{s}''=\frac{1}{(1+\cosh(r))^2} g_r.
	$$
	In particular, as $\Sigma_s'$ is $(m-1)$-dimensional, one has
	$$
	|\Sigma_s'|_{\Real}\leq \frac{|\Sigma_r|_{\mathbb{CH}}}{(1+\cosh(r))^{m-1}}.
	$$
	The definition of $C^1$-regular asymptotic boundary ensures that
	$$
	\lim_{s\to 1} |\Sigma_s'|_{\Real}=|\partial \Sigma' |_{\Real}=|\partial \Sigma' |_{\mathbb{S}}= Vol_{\partial_\infty \mathbb{CH}}(\partial_\infty \Sigma; p_0).
	$$
	As $s\to 1$,  $r\to \infty$, and so, using $\lim_{r\to \infty} \frac{\sinh(r)}{1+\cosh(r)}=1$, we conclude
	\begin{align*}
		Vol_{\partial_\infty \mathbb{CH}}(\partial_\infty \Sigma; p_0) &\leq 
		\liminf_{r\to \infty}   \frac{|\Sigma_r|_{\mathbb{CH}}}{(1+\cosh(r))^{m-1}}.
	\end{align*}

	To see the second claim, observe that Lemma \ref{WeakAsy2StrongAsyLem} implies that the new hypotheses on  $\Sigma$ encompasses the previous ones and ensures that $\Sigma'$ has the property that $\mathbf{T}^\top$ on $\Sigma'_s$ is of size $O((1-s)^2)$.  It follows that in this case,  \eqref{ComparisionMetricEqn} satisfies
	$$
g_s''=\frac{4}{(1-s^2)^2} g_s' + \frac{4s^4}{(1-s^2)^4} (i_{\Sigma'_s}^*\theta)^2\leq \frac{4}{(1-s^2)^2} g_s' +C g_s'
$$
for some constant $C>0$.
Hence, up to increasing $C$
	$$
(1+C e^{-2r}) \tilde{\Upsilon}_{p_0}^*g_{s}'\geq \frac{1-Ce^{-2r}}{(1+\cosh(r))^2} g_r.
$$
This means that up to increasing $C$, 
$$
	|\Sigma_s'|_{\Real}\geq \frac{(1-C e^{-2r})|\Sigma_r|_{\mathbb{CH}}}{(1+\cosh(r))^{m-1}}.
$$
The second claim then follows as before by taking $s\to 1$ and $r\to \infty$.
\end{proof}

Theorem \ref{main2} is a consequence of Theorem \ref{main1} and the following proposition:
\begin{prop}\label{LimitProp}
	Let $\Sigma\subset \mathbb{CH}^{n+1}$ be an $m$-dimensional submanifold that is weakly $C^1$-asymptotically regular and weakly horizontal. Then
	
	\begin{equation}
	\liminf_{t\to -\infty} \int_{\Sigma} \Phi^{t_0, p_0}_{m,1} (t,p) dVol_{\Sigma}(p)\geq\frac{Vol_{\partial_\infty \mathbb{CH}}(\partial_\infty \Gamma; p_0)}{|\mathbb{S}^{m-1}|_{\Real}}. \label{TimeDecayEqn}
	\end{equation}
	If $\Sigma$ is $C^1$-asymptotically regular and strongly asymptotically horizontal, then
	\begin{equation}
\lim_{t\to -\infty} \int_{\Sigma} \Phi^{t_0, p_0}_{m,1} (t,p) dVol_{\Sigma}(p)=\frac{Vol_{\partial_\infty \mathbb{CH}}(\partial_\infty \Gamma; p_0)}{|\mathbb{S}^{m-1}|_{\Real}}. \label{TimeDecayEqn2}
\end{equation}
\end{prop}
\begin{proof}
	Let $\Sigma_r=\Sigma\cap \partial B_{r}^{\mathbb{CH}}(p_0)$.  Observe that, by the definition of being weakly $C^1$-asymptotically regular there is an $R_0>0$ such that, for $r\geq R_0$, $\Sigma $ meets $\partial B_r^{\mathbb{CH}}(p_0)$ transversally and so $\Sigma_r$ is a smooth $(m-1)$-dimensional submanifold of  $\partial B_r^{\mathbb{CH}}(p_0)$.  
	
By Lemma \ref{BoundaryLem}, for any $\epsilon>0$, there is an $R_\epsilon>R_0$ so, for $r>R_\epsilon$,
	$$
	(1-\epsilon) { Vol_{\partial_\infty \mathbb{CH}}(\partial_\infty \Sigma; p_0)}\leq  \frac{ |\Sigma_r|_{\mathbb{CH}}}{{\sinh^{m-1}(r)}}.
	$$
	
	As $|\nabla_{\mathbb{CH}} \rho|\leq 1$, using the co-area formula, for $R>R_\epsilon$, one has
	\begin{align*}
		(1-\epsilon) Vol_{\partial_\infty\mathbb{CH}}(\partial_\infty \Sigma; p_0) &
		\int_R^\infty K_{m,1}(t_0-t, r)  \sinh^{m-1}(r) dr\\
		& \leq \int_{\Sigma\setminus B_R^{\mathbb{CH}}(p_0)} \Phi_{m,1}^{t_0,p_0}(t,p) dVol_\Sigma(p).
	\end{align*}
As $K_{m,1}(t,r)=K_{m}(t,r)$, it follows from the proof of  \cite[Proposition 4.2]{BernsteinHypEntropy} that
$$
\lim_{t\to -\infty} 	\int_R^\infty K_{m,1}(t_0-t, r)  \sinh^{m-1}(r) dr= |\mathbb{S}^{m-1}|_{\Real}^{-1}.
$$
 
Hence,
	\begin{align*}
		(1-\epsilon) \frac{ Vol_{\partial_\infty \mathbb{CH}}(\partial_\infty \Sigma; p_0) }{|\mathbb{S}^{m-1}|_{\Real}}&\leq \liminf_{t\to -\infty} \int_{\Sigma\setminus B_R^{\mathbb{CH}}(p_0)} \Phi_{m,1}^{t_0,p_0}(t,p) dVol_\Sigma(p)\\
		&=\liminf_{t\to -\infty} \int_{\Sigma} \Phi_{m,1}^{t_0,p_0}(t,p) dVol_\Sigma(p)
	\end{align*}
	where the second equality again uses \eqref{TimeDecayEqn}.
	Sending $\epsilon \to 0$ yields,
	$$
	\frac{ Vol_{\partial_\infty \mathbb{CH}}(\partial_\infty \Sigma; p_0) }{|\mathbb{S}^{m-1}|_{\Real}}\leq \liminf_{t\to -\infty} \int_{\Sigma} \Phi_{m,1}^{t_0,p_0}(t,p) dVol_\Sigma(p),
	$$
	which verifies the first claim. 
	
	Suppose $\Sigma$ is $C^1$-asymptotically regular and strongly asymptotically horizontal.  By Lemma \ref{WeakAsy2StrongAsyLem}, in the modified Bergman compactification, the compactified surface meets the ideal boundary orthogonally.  It follows that if $\rho=\dist_{\mathbb{CH}}(p,p_0)$, then 
$$
\lim_{p\to \infty} |\nabla_\Sigma \rho|=1.
$$
In particular, there is an $R_\epsilon'>0$ sufficiently large such that for $p\in \Sigma\setminus B_{R_\epsilon'}^{\mathbb{CH}}(p_0)$, 
$$
1\leq \frac{1}{|\nabla_\Sigma \rho|}\leq 1+\epsilon.
$$
Appealing to Lemma \ref{BoundaryLem}, up to increasing $R_\epsilon'>0$ one has for $r>R_{\epsilon}'$, 
	$$
 \frac{ |\Sigma_r|_{\mathbb{CH}}}{{\sinh^{m-1}(r)}}\leq (1+\epsilon) { Vol_{\partial_\infty \mathbb{CH}}(\partial_\infty \Sigma; p_0)}. 
$$

Hence, for $R>R_{\epsilon}'$, one has
\begin{align*}
	\int_{\Sigma\setminus B_R^{\mathbb{CH}}(p_0)} & \Phi_{m,1}^{t_0,p_0}(t,p) dVol_\Sigma(p)\\
	&\leq   (1+\epsilon)^2 Vol_{\partial_\infty\mathbb{CH}}(\partial_\infty \Sigma; p_0) \int_R^\infty K_{m,1}(t_0-t, r)  \sinh^{m-1}(r) dr.
\end{align*}
Arguing as above, 
$$
\limsup_{t\to -\infty} \int_{\Sigma} \Phi_{m,1}^{t_0,p_0}(t,p) dVol_\Sigma(p)\leq
\frac{ Vol_{\partial_\infty \mathbb{CH}}(\partial_\infty \Sigma; p_0) }{|\mathbb{S}^{m-1}|_{\Real}}.
$$
This proves that
$$
\lim_{t\to -\infty} \int_{\Sigma} \Phi_{m,1}^{t_0,p_0}(t,p) dVol_\Sigma(p)=
\frac{ Vol_{\partial_\infty \mathbb{CH}}(\partial_\infty \Sigma; p_0) }{|\mathbb{S}^{m-1}|_{\Real}}
$$
verifying the second claim.  
\end{proof}

We may now prove Theorem \ref{main2}.
\begin{proof}[Proof of Theorem \ref{main2}]
	By definition, for any fixed $p_0\in \mathbb{CH}^{n+1}$,
	$$
	\lambda_{\mathbb{CH}}[\Sigma]\geq \limsup_{t\to -\infty} \int_{\Sigma} \Phi^{0, p_0}_{m,1} (t,p) dVol_{\Sigma}(p).
	$$
	Hence, Proposition \ref{LimitProp} implies that
	$$
	\lambda_{\mathbb{CH}}[\Sigma]\geq  \frac{Vol_{\partial_\infty \mathbb{CH}}(\partial_\infty \Sigma; p_0)}{|\mathbb{S}^{m-1}|_{\Real}}.
	$$
	Taking the supremum over $p_0\in\mathbb{CH}^{n+1}$ and using $\Gamma=\partial_\infty\Sigma$ yields, 
	$$
	\lambda_{\mathbb{CH}}[\Sigma]\geq\frac{\lambda_{CR}[\Gamma]}{|\mathbb{S}^{m-1}|_{\Real}}.
	$$
	This proves the first claim. 
	
	To see the second claim, observe that if $\Sigma$ is weakly $C^2$-asymptotically regular, weakly asymptotically horizontal and minimal, then we may apply Theorem \ref{main1} to see that $\Sigma$ is $C^1$-asymptotically regular and strongly asymptotically horizontal. In particular, Proposition \ref{LimitProp} implies
	$$
	\lim_{t\to -\infty} \int_{\Sigma} \Phi^{0, p_0}_{m,1} (t,p) dVol_{\Sigma}(p)=\frac{Vol_{\partial_\infty \mathbb{CH}}(\partial_\infty \Sigma; p_0)}{|\mathbb{S}^{m-1}|_{\Real}}.
	$$
	 Furthermore,  $\Sigma$ may be thought of as a static solution of mean curvature flow that, by Lemma \ref{BoundaryLem}, and the co-area formula has exponential volume growth.  Hence, by \cite[Theorem 1.1]{BernsteinBhattacharya}, for all $\tau>0$, one has
	\begin{align*}
		\int_{\Sigma} \Phi^{0, p_0}_{m,1} (-\tau,p) dVol_{\Sigma}(p) &\leq \lim_{t\to -\infty} \int_{\Sigma} \Phi^{0, p_0}_{m,1} (t-\tau,p) dVol_{\Sigma}(p)\\
		&= \frac{Vol_{\partial_\infty \mathbb{CH}}(\partial_\infty \Gamma; p_0)}{|\mathbb{S}^{m-1}|_{\Real}}\leq \frac{\lambda_{CR}[\Gamma]}{|\mathbb{S}^{m-1}|_{\Real}}.
	\end{align*}
	Taking the supremum over $\tau>0$ and $p_0\in \mathbb{CH}^{n+1}$ yields
	$$
	\lambda_{\mathbb{CH}}[\Sigma]\leq \frac{\lambda_{CR}[\Gamma]}{|\mathbb{S}^{m-1}|_{\Real}}.
	$$
	Combined with the first claim this completes the proof.
\end{proof}

\appendix
\section{Huisken Monotonicity in $\mathbb{CH}^{n+1}$}
In $\mathbb{CH}^{n+1}$, the monotonicity formula from \cite{BernsteinBhattacharya} has a particularly simple form.  We first record without proof the computation of the Hessian of a radial function.
\begin{lem}\label{HessDistLem}
	Let $\rho(p)=\dist_{\mathbb{CH}}(p,p_0)$ and suppose that $F:[0, \infty)\to \Real$ is a $C^2$ function.   If $f(p)=F(\rho(p))$ on $\mathbb{CH}^{n+1}\setminus \set{p_0}$, then
\begin{align*}
	\nabla_{\mathbb{CH}}^2 f &= \coth (\rho) F'(\rho) g_{\mathbb{CH}} +\left( F''(\rho)-  \coth (\rho) F'(\rho)\right) d\rho^2 + \\
	&+ \tanh(\rho) F'(\rho)  (d\rho \circ J_{\mathbb{CH}})^2.
\end{align*}
\end{lem}
We may now specialize some of the conclusions of \cite{BernsteinBhattacharya} to $\mathbb{CH}^{n+1}$.
\begin{prop}\label{MonotonicityFormulaProp}
	Suppose that $\set{\Sigma_t}_{t\in [0,T)}$ is a mean curvature flow in $\mathbb{CH}^{n+1}$ of $m$-dimensional submanifolds with exponential volume growth. For any $t_0>0$, $p_0\in \mathbb{CH}^{n+1}$, and $t\in (0,\min\set{t_0, T})$ one has
	\begin{align*}
		\frac{d}{dt} \int_{\Sigma_t}  \Phi_{m,1}^{t_0,p_0} dV_{\Sigma_t} =- \int_{\Sigma_t}\left( \left| \frac{\nabla^\perp_{\Sigma_t} \Phi_{m,1}^{t_0,p_0}}{\Phi_{m,1}^{t_0,p_0}} -\mathbf{H}_{\Sigma_t}\right|^2 +Q_{m,1}^{t_0,p_0}(t,x, N_x \Sigma_t) \right) \Phi_{m,1}^{t_0,p_0} dV_{\Sigma_t}.
	\end{align*}
	Here,
\begin{align*}
Q_{m,1}^{t_0,p_0}(t,x, N_x \Sigma_t)&=\left(\log K_{m,1}''(t,\rho)-\coth \rho \log K_{m,1}'(t,\rho)\right) |\nabla^\perp_\Sigma \rho|^2 \\
&-\tanh \rho |(J_{\mathbb{CH}}(\nabla_{\mathbb{CH}}\rho))^\top|^2    \log K_{m,1}'(t,\rho)\geq 0.
\end{align*}
Moreover, this inequality is strict somewhere unless $\Sigma_t$ is an isotropic cone over $p_0$.
\end{prop}
\begin{proof}
   By \cite[Proposition 5.1]{BernsteinBhattacharya} to obtain the above formulas it is enough to compute
   \begin{align*}
   Q_{m,1}^{t_0, p_0}(t,x,N_x \Sigma_t) &= \sum_{i=1}^{k} \nabla^2_{\mathbb{CH}} \Phi_{m,1}^{t_0, p_0} (E_i, E_i) \\
   &+\left((m-1) \coth(\rho) -\Delta_{\mathbb{CH}} \rho\right) \partial_\rho \log \Phi_{m,1}^{t_0, p_0}
  \end{align*}
  where $k=2n+2-m$ and $E_1, \ldots, E_{k}$ is an orthonormal basis of $N_x\Sigma_t$.
  By Lemma \ref{HessDistLem},
 $$
 \Delta_{\mathbb{CH}} \rho= (2n+1) \coth(\rho) +\tanh(\rho) 
 $$
 while  
\begin{align*}
 \sum_{i=1}^{k} \nabla^2_{\mathbb{CH}} \Phi_{m,1}^{t_0, p_0} (E_i, E_i) &= k \coth(\rho) \log K_{m,1}' +(\log K_{m,1}''- \coth \rho \log K_{m,1}') |\nabla^\perp_\Sigma \rho|^2\\
 &+\tanh(\rho)  \log K_{m,1}' |(J_{\mathbb{CH}}(\nabla_{\mathbb{CH}} \rho))^\perp|^2.
\end{align*}
The claimed formula follows from this.  The inequality is an immediate consequence of properties of  $K_{m,1}$ -- see \cite[Section  5]{BernsteinBhattacharya}.  The strictness of the inequality unless $\Sigma_t$ is a geodesic cone over $p_0$ follows from \cite[Lemma 5.3]{BernsteinBhattacharya}.  Finally, if $\Sigma$ is a geodesic cone, then $\nabla_{\mathbb{CH}} \rho$ is tangent to $\Sigma_t$.  For the inequality to not be strict one must have $J_{\mathbb{CH}}( \nabla_{\mathbb{CH}} \rho)$ orthogonal to $\Sigma_t$.  If $\Upsilon: \mathbb{CH}^{n+1}\to \mathbb{B}^{2n+2}$ is a Bergman compactification with $\Upsilon(p_0)=\mathbf{0}$, then this condition is equivalent to $\Upsilon(\Sigma_t)$, which is a cone over $\mathbf{0}$,  being orthogonal to $\mathbf{T}$ that can only occur if the cone is isotropic as this is equivalent to the link of the cone being horizontal. 
   \end{proof}
\section{Geometric Computations}
%
%

\subsection{Rank one deformation} \label{RankOneSec}
Let $(M,g)$ be a Riemannian manifold, $\tau$ a smooth one-form on $M$ and $\alpha$ a smooth  function.  Suppose that $|\tau|_g^2 \alpha>-1$.  We may then define a new Riemannian metric
$$
h=g+\alpha \tau\otimes \tau.
$$
Let $\nabla^g$ denote the Levi-Civita connection of $g$ and $\nabla^h$ be the Levi-Civita connection of $h$.
Let $C(X,Y)$ be the $(1,2)$ tensor field given by
$$
C(X,Y)=\nabla^h_{X} Y-\nabla^g_X Y.
$$
Using the Koszul formula we compute that
\begin{align*}
 h(C(X,Y),Z)=	g(C(X,Y),Z)+\alpha \tau(C(X,Y))\tau(Z)= c(X,Y,Z)
\end{align*}
where $c(X,Y,Z)$ satisfies
\begin{align*}
	c(X,Y,Z) &= \frac{1}{2} (\nabla^g_{X}\alpha \tau)(Y)\tau(Z)+\frac{1}{2} (\nabla^g_{Y}\alpha \tau)(X) \tau(Z)+\frac{1}{2} \alpha \nabla_X \tau(Z) \tau(Y)\\
	&+\frac{1}{2} \alpha \nabla_Y \tau(Z) \tau(X)-\frac{1}{2}  (\nabla_Z^g \alpha \tau)(X)\tau(Y)-\frac{1}{2} \alpha\nabla_Z^g  \tau(Y)\tau(X).
\end{align*}
Let $\mathbf{T}$ be the vector field such that
$$
g(\mathbf{T}, Z)=\tau(Z).
$$
Choose orthonormal vectors $E_1, \ldots, E_{n}$ that are orthogonal to $\mathbf{T}$ and so $\set{E_1, \ldots, E_n, \mathbf{T}}$ spans $T_p M$.
It follows that
$$
C(X,Y)=\frac{|\mathbf{T}|^{-2}_g}{1+\alpha|\mathbf{T}|_g^2} c(X,Y, \mathbf{T}) \mathbf{T}+\sum_{i=1}^n c(X,Y, E_i) E_i
$$
where $c$ and $d_i$ are symmetric $(0,2)$ tensor fields that satisfy
\begin{align*}
	c(X,Y, \mathbf{T})=\frac{1}{2} \big(g(\nabla^g_X \alpha \mathbf{T}, Y)|\mathbf{T}|^{2}_g+g(\nabla^g_Y \alpha\mathbf{T}, X)|\mathbf{T}|^{2}_g+\alpha  g(\nabla_X^g \mathbf{T}, \mathbf{T})g(\mathbf{T}, Y)\\ 
	+\alpha  g(\nabla_Y^g \mathbf{T}, \mathbf{T})g(\mathbf{T}, X)-g(\nabla_{\mathbf{T}}^g \alpha \mathbf{T}, X)g(\mathbf{T}, Y)-\alpha g(\nabla_{\mathbf{T}}^g  \mathbf{T}, Y)g(\mathbf{T},X) \big)
\end{align*}
and
\begin{align*}
	c(X,Y, E_i)= \frac{1}{2} \alpha g(\nabla_X^g \mathbf{T}, E_i) g(\mathbf{T},Y)+\frac{1}{2} \alpha g(\nabla_Y^g \mathbf{T}, E_i) g(\mathbf{T},X)\\-\frac{1}{2} g(\nabla_{E_i}^g \alpha \mathbf{T}, X) g(\mathbf{T},Y)-\frac{1}{2} g(\nabla_{E_i}^g \alpha \mathbf{T}, Y) g(\mathbf{T},X).
\end{align*}

\subsection{Second fundamental form}\label{CurvatureAppendix}
We compute the relationship between the second fundamental form of a submanifold in a Riemannian manifold $(M,g)$ with the same quantity with respect to a metric $h$ obtained as a rank one deformation of $g$.

Let $\Sigma\subset M$ be a submanifold.  Observe that for $X$ and $Y$ tangent to $\Sigma$ one has
\begin{align*}
h(X,Y)&=g(X,Y)+\tau(X)\tau(Y)=g(X,Y)+g(\mathbf{T}, X)g(\mathbf{T}, Y),\\
&h_{\Sigma}(X,Y)=g_\Sigma(X,Y)+g_\Sigma(\mathbf{T}^\top, X)g_\Sigma(\mathbf{T}^\top, Y)
\end{align*}
where $\mathbf{T}^\top$ is the $g$-tangential component of $\mathbf{T}$.  Denote by $\nabla^{\Sigma,g}$ and $\nabla^{\Sigma,h}$ the induced connections on $\Sigma$ from $\nabla^g$ and $\nabla^h$ respectively.  
Let $\mathbf{T}^N$ be the $g$-normal component of $\mathbf{T}$ to $\Sigma$ and let $\mathbf{T}^{\hat{N}}$ be the $h$-normal component.

We have
$$
\mathbf{T}^{\hat{N}}= \mathbf{T}^N-\frac{|\mathbf{T}^N|_{g}^2}{1+|\mathbf{T}^\top|^2_g} \mathbf{T}^\top=\mathbf{T}-\frac{1+|\mathbf{T}|_g^2}{1+|\mathbf{T}^\top|^2_g} \mathbf{T}^\top.
$$
Now choose $N_1,\ldots, N_l$ to be $g$-unit length tangent vectors that are $g$-orthogonal to $\mathbf{T}$ and $\Sigma$.  These are also of $h$-unit length and are $h$-orthogonal to $\mathbf{T}$ and $\Sigma$.  Likewise, let $E_1, \ldots, E_k$ be $g$-unit length vectors tangent to $\Sigma$ and $g$-orthogonal to $\mathbf{T}$. 
For a vector $\mathbf{V}\in T_p M$ with $p\in \Sigma$, let us denote 
$$
\mathbf{V}^{\tilde{N}}=\mathbf{V}^N-g(\mathbf{V}^N, \mathbf{T}) \frac{\mathbf{T}^{N}}{|\mathbf{T}^{N}|_g^2}=\sum_{i=1}^l g(\mathbf{V}, N_i)N_i.
$$
So $\mathbf{V}^{\tilde{N}}$ is both $h$ and $g$-orthogonal to $\Sigma$ and $\mathbf{T}$. In particular, 
$$g(\mathbf{V}^{\tilde{N}}, \mathbf{T}^{\hat{N}})=g(\mathbf{V}^{{N}}, \mathbf{T}^{\hat{N}})=g(\mathbf{V}^{\tilde{N}}, \mathbf{T})=0.$$

\begin{prop}\label{RankOne2ndFFProp}
	Let $(M,g)$ be a Riemannian manifold and $h$ a Riemannian metric obtained as a rank one deformation of $g$ by the one form $\tau=g(\mathbf{T}, \cdot)$.  Suppose that
	$$
	\nabla_Z^g \mathbf{T} = -\mathbf{a}(Z)
	$$
	where  $\mathbf{a}$ is a $(1,1)$-tensor satisfying
	$$
	g(\mathbf{a}(X), Y)=-g(X, \mathbf{a}(Y)). 
	$$
For a submanifold $\Sigma\subset M$ one has the following relationship between the second fundamental form of $\Sigma$ with respect to $g$ and $h$
$$
(	\mathbf{A}_\Sigma^h(X,Y))^{\tilde{N}}= (\mathbf{A}_\Sigma^g(X,Y))^{\tilde{N}}- g(\mathbf{T}, Y) (\mathbf{a}(X))^{{\tilde{N}}} -g(\mathbf{T}, X) (\mathbf{a}(Y))^{{\tilde{N}}}
$$
and
\begin{align*}
	g(\mathbf{A}_{\Sigma}^h(X,Y), \mathbf{T}^{\hat{N}})&=g( \mathbf{A}_{\Sigma}^g(X,Y), \mathbf{T})\\
	&-\frac{1+|\mathbf{T}^\top|^2_g}{1+|\mathbf{T}|_g^2}\big( g(\mathbf{T}^{\hat{N}}, \mathbf{a}(X))g(\mathbf{T}, Y)	+  g(\mathbf{T}^{\hat{N}},\mathbf{a}(Y))g(\mathbf{T}, X)\big).  
\end{align*}
\end{prop}

\begin{proof}
The computations of Section \ref{RankOneSec} yield, for $X$ and $Y$ tangent to $\Sigma$, and $N_j$ $g$-orthogonal to $\Sigma$ and $\mathbf{T}$,
\begin{align*}
	h(\nabla_X^h Y&, N_j)=g(\nabla_X^g Y, N_j)+c(X,Y, N_j)\\
	&= g(\mathbf{A}_\Sigma^g(X,Y), N_j)+\frac{1}{2} g(\nabla_X^g \mathbf{T}, N_j) g(\mathbf{T},Y)+\frac{1}{2} g(\nabla_Y^g \mathbf{T}, N_j) g(\mathbf{T},X)\\
	&-\frac{1}{2} g(\nabla_{N_j}^g  \mathbf{T}, X) g(\mathbf{T},Y)-\frac{1}{2} g(\nabla_{N_j}^g  \mathbf{T}, Y) g(\mathbf{T},X).                    
\end{align*}
The additional hypothesis on $\mathbf{T}$ and $\mathbf{a}$ imply
\begin{align*}
		h(\nabla_X^h Y, N_j)&= g(\mathbf{A}_\Sigma^g(X,Y), N_j)- g(\mathbf{a}(X), N_j) g(\mathbf{T},Y)- g(\mathbf{a}(Y), N_j) g(\mathbf{T},X).
\end{align*}
This immediately yields the first formula.

Observe that
\begin{align*}
	h(\nabla_X^h Y, \mathbf{T}^{\hat{N}})
	&=h(\nabla_X^h Y, \mathbf{T})- \frac{1+|\mathbf{T}|_g^2}{1+|\mathbf{T}^\top|_g^2} h(\nabla_X^{\Sigma,h} Y, \mathbf{T}^\top).
\end{align*}
It directly follows from Section \ref{RankOneSec} that
\begin{align*}
	h(\nabla_X^h Y, \mathbf{T})&=(1+|\mathbf{T}|^2_g) g(\nabla_X^g Y, \mathbf{T})+c(X,Y, \mathbf{T})
\end{align*}
where
\begin{align*}
	c(X,Y,\mathbf{T})&=\frac{1}{2} \big(g(\nabla^g_X  \mathbf{T}, Y)|\mathbf{T}|^{2}_g+g(\nabla^g_Y \mathbf{T}, X)|\mathbf{T}|^{2}_g+  g(\nabla_X^g \mathbf{T}, \mathbf{T})g(\mathbf{T}, Y)\\ 
	&	+  g(\nabla_Y^g \mathbf{T}, \mathbf{T})g(\mathbf{T}, X)-g(\nabla_{\mathbf{T}}^g  \mathbf{T}, X)g(\mathbf{T}, Y)- g(\nabla_{\mathbf{T}}^g  \mathbf{T}, Y)g(\mathbf{T},X) \big).
\end{align*}	
Likewise, treating $h_\Sigma$ as a rank one deformation of $g_\Sigma$ one has
\begin{align*}
	h(\nabla_X^{\Sigma,h} Y, \mathbf{T}^\top)&=(1+|\mathbf{T}^\top|_g^2) g(\nabla_X^{\Sigma,g} Y, \mathbf{T}^\top)+c_\Sigma(X,Y,\mathbf{T}^\top)\\
 &=	(1+|\mathbf{T}^\top|_g^2) g(\nabla_X^{\Sigma,g} Y, \mathbf{T})+c_\Sigma(X,Y,\mathbf{T}^\top)
\end{align*}
where the definition and properties of the connection yield
\begin{align*}
 c_\Sigma&(X,Y, \mathbf{T}^\top)\\
 &  =\frac{1}{2} \big(g(\nabla^{\Sigma,g}_X  \mathbf{T}^\top, Y)|\mathbf{T}^\top|^{2}_g+g(\nabla^{\Sigma,g}_Y \mathbf{T}^\top, X)|\mathbf{T}^\top|^{2}_g+  g(\nabla_X^{\Sigma,g} \mathbf{T}^\top, \mathbf{T}^\top)g(\mathbf{T}^\top, Y)\\ 
 &	+  g(\nabla_Y^{\Sigma,g} \mathbf{T}^\top, \mathbf{T}^\top)g(\mathbf{T}^\top, X)-g(\nabla_{\mathbf{T}^\top}^{\Sigma,g}  \mathbf{T}^\top, X)g(\mathbf{T}^\top, Y)- g(\nabla_{\mathbf{T}^\top}^{\Sigma,g}  \mathbf{T}^\top, Y)g(\mathbf{T}^\top,X) \big)\\
&= \frac{1}{2} \big(g(\nabla^g_X  \mathbf{T}^\top, Y)|\mathbf{T}^\top|^{2}_g+g(\nabla^g_Y \mathbf{T}^\top, X)|\mathbf{T}^\top|^{2}_g+ g(\nabla_X^{g} \mathbf{T}^\top, \mathbf{T}^\top)g(\mathbf{T}^\top, Y)\\ 
&	+  g(\nabla_Y^{g} \mathbf{T}^\top, \mathbf{T}^\top)g(\mathbf{T}^\top, X)-g(\nabla_{\mathbf{T}^\top}^{g}  \mathbf{T}^\top, X)g(\mathbf{T}^\top, Y)- g(\nabla_{\mathbf{T}^\top}^{g}  \mathbf{T}^\top, Y)g(\mathbf{T}^\top,X) \big).
\end{align*}
For tangent vectors $V,W$ one has
$$
g(\nabla_V^g \mathbf{T}^\top, W) =g(\nabla_V^g (\mathbf{T}-\mathbf{T}^N), W)=g(\nabla_V^g \mathbf{T}, W)+g(\mathbf{T}, A_{\Sigma}^g(V,W)).
$$
Hence,
\begin{align*}
c_\Sigma&(X,Y, \mathbf{T}^\top)\\
&=\frac{1}{2} \big(g(\nabla^g_X  \mathbf{T}, Y)|\mathbf{T}^\top|^{2}_g+g(\nabla^g_Y \mathbf{T}, X)|\mathbf{T}^\top|^{2}_g+  g(\nabla_X^{g} \mathbf{T}, \mathbf{T}^\top)g(\mathbf{T}, Y)\\ 
&	+  g(\nabla_Y^{g} \mathbf{T}, \mathbf{T}^\top)g(\mathbf{T}, X)-g(\nabla_{\mathbf{T}^\top}^{g}  \mathbf{T}, X)g(\mathbf{T}, Y)- g(\nabla_{\mathbf{T}^\top}^{g}  \mathbf{T}, Y)g(\mathbf{T},X) \big)\\
&+\frac{1}{2} \big(g( \mathbf{T}, \mathbf{A}_{\Sigma}(X, Y))|\mathbf{T}^\top|^{2}_g+g( \mathbf{T}, \mathbf{A}_{\Sigma}(Y, X))|\mathbf{T}^\top|^{2}_g+  g( \mathbf{T}, \mathbf{A}_{\Sigma}(X,\mathbf{T}^\top))g(\mathbf{T}, Y)\\ 
&	+  g(\mathbf{T}, \mathbf{A}_{\Sigma}(Y, \mathbf{T}^\top))g(\mathbf{T}, X)-g(  \mathbf{T},\mathbf{A}_\Sigma( X,\mathbf{T}^\top))g(\mathbf{T}, Y))- g(\mathbf{T}, \mathbf{A}_{\Sigma}(Y, \mathbf{T}^\top))g(\mathbf{T},X) \big)\\
&=\frac{1}{2} \big(g(\nabla^g_X  \mathbf{T}, Y)|\mathbf{T}^\top|^{2}_g+g(\nabla^g_Y \mathbf{T}, X)|\mathbf{T}^\top|^{2}_g+  g(\nabla_X^{g} \mathbf{T}, \mathbf{T}^\top)g(\mathbf{T}, Y)\\ 
&	+  g(\nabla_Y^{g} \mathbf{T}, \mathbf{T}^\top)g(\mathbf{T}, X)-g(\nabla_{\mathbf{T}^\top}^{g}  \mathbf{T}, X)g(\mathbf{T}, Y)- g(\nabla_{\mathbf{T}^\top}^{g}  \mathbf{T}, Y)g(\mathbf{T},X) \big)\\
&+g( \mathbf{T}^N, \mathbf{A}_{\Sigma}(X, Y))|\mathbf{T}^\top|^{2}_g.
\end{align*}
Putting everything together yields
\begin{align*}
	h(\nabla_X^h Y,& \mathbf{T}^{\hat{N}})=	(1+|\mathbf{T}|^2_g)g(\mathbf{T}, \mathbf{A}_\Sigma^g(X,Y)+c(X,Y, \mathbf{T})-\frac{1+|\mathbf{T}|_g^2}{1+|\mathbf{T}^\top|_g^2}c_\Sigma(X,Y, \mathbf{T}^\top)\\
	&=\frac{1+|\mathbf{T}|^2_g}{1+|\mathbf{T}^\top|_g^2}	g(\mathbf{T}, \mathbf{A}_{\Sigma}^g (X,Y)) +\frac{|\mathbf{T}^{N}|^{2}_g}{2(1+|\mathbf{T}^\top|^2_g)} \big(g(\nabla^g_X  \mathbf{T}, Y)+g(\nabla^g_Y \mathbf{T}, X)\big)\\
	&+ \frac{1}{2}\big( g(\nabla_X^g \mathbf{T}, \mathbf{T}^{\hat{N}})g(\mathbf{T}, Y)	+  g(\nabla_Y^g \mathbf{T}, \mathbf{T}^{\hat{N}})g(\mathbf{T}, X)\\
	&-g(\nabla_{\mathbf{T}^{\hat{N}}}^g  \mathbf{T}, X)g(\mathbf{T}, Y)- g(\nabla_{\mathbf{T}^{\hat{N}}}^g  \mathbf{T}, Y)g(\mathbf{T},X) \big).
\end{align*}
Using the hypotheses on $\mathbf{T}$ and $\mathbf{a}$ yields
\begin{align*}
	h(\nabla_X^h Y, \mathbf{T}^{\hat{N}})&=\frac{1+|\mathbf{T}|^2_g}{1+|\mathbf{T}^\top|_g^2}	g(\mathbf{T}^{\hat{N}}, \mathbf{A}_{\Sigma}^g (X,Y)) \\
	&-\big( g(\mathbf{a}(X), \mathbf{T}^{\hat{N}})g(\mathbf{T}, Y)	+  g(\mathbf{a}(Y), \mathbf{T}^{\hat{N}})g(\mathbf{T}, X)\big)
\end{align*}
where we used $g(\mathbf{T}, \mathbf{A}_{\Sigma}^g (X,Y))=g(\mathbf{T}^{\hat{N}}, \mathbf{A}_{\Sigma}^g (X,Y))$. 

By definition,
$$
h(\mathbf{A}_\Sigma^h(X,Y), \mathbf{T})=h(\mathbf{A}_\Sigma^h(X,Y), \mathbf{T}^{\hat{N}})=	h(\nabla_X^h Y, \mathbf{T}^{\hat{N}}).
$$
This means
\begin{align*}
h(\nabla_X^h Y, \mathbf{T}^{\hat{N}})&=g(\mathbf{A}_\Sigma^h(X,Y), \mathbf{T}^{\hat{N}})+\frac{|\mathbf{T}^N|_g^2}{1+|\mathbf{T}^\top|_g^2}  g(\mathbf{A}_\Sigma^h(X,Y), \mathbf{T}^{\hat{N}})\\
&=\frac{1+|\mathbf{T}|_g^2}{1+|\mathbf{T}^\top|_g^2}g(\mathbf{A}_\Sigma^h(X,Y), \mathbf{T}^{\hat{N}}) .
\end{align*}
Combining this with the previous computation yields the second formula.
\end{proof}

\begin{cor}\label{MCRankOneCor}
	Likewise, the relationship between the mean curvatures is given by
	\begin{equation*}
		(	\mathbf{H}_\Sigma^h)^{\tilde{N}}= (\mathbf{H}_\Sigma^g)^{\tilde{N}}- \frac{(\mathbf{A}_\Sigma^g(\mathbf{T}^\top, \mathbf{T}^\top))^{\tilde{N}} +2  (\mathbf{a}(\mathbf{T}^\top))^{\tilde{N}}}{1+|\mathbf{T}^\top|_g^2} 
	\end{equation*}
	and
	\begin{align*}
		g(\mathbf{H}_{\Sigma}^h, \mathbf{T}^{\hat{N}})&=g( \mathbf{H}_{\Sigma}^g, \mathbf{T})- \frac{g(\mathbf{A}_\Sigma^g(\mathbf{T}^\top, \mathbf{T}^\top), \mathbf{T})}{1+|\mathbf{T}^\top|_g^2}-\frac{2g(\mathbf{a}(\mathbf{T}^\top), \mathbf{T})}{1+|\mathbf{T}|_g^2}.  
	\end{align*}
\end{cor}
\begin{proof}

Let us now compute the mean curvature.
At a point, $p\in \Sigma$, where $\mathbf{T}^\top(p)=\mathbf{0}$, we immediately see that
$$
\mathbf{H}_{\Sigma}^h(p) =  \mathbf{H}_{\Sigma}^g(p).
$$
This verifies both claimed formulas in this case.

When $\mathbf{T}^\top(p)\neq 0$,  we may choose a $g$-orthonormal  basis of $T_p \Sigma$ of the form $\set{E_1'= |\mathbf{T}^\top|_g^{-1}\mathbf{T}^\top,  E_2, \ldots, E_k}$. 
Setting, 
$$ 
E_1=\frac{\mathbf{T}^\top}{|\mathbf{T}^\top|_g\sqrt{1+|\mathbf{T}^\top|_g^2}}=(1-\beta|\mathbf{T}^\top|_g^2) E_1' \mbox{ where } \beta= \frac{1}{1+|\mathbf{T}^\top|_g^2+ \sqrt{ 1+|\mathbf{T}^\top|_g^2}},
$$
we obtain an $h$-orthonormal basis $\set{E_1, \ldots, E_k}$. 
Clearly, $g(\mathbf{T}, E_j)=0$ for $2\leq j\leq k$ and
$$
2\beta-|\mathbf{T}^\top|_g^2\beta^2= \frac{1}{1+|\mathbf{T}^\top|_g^2}.
$$
Using this $h$-orthonormal basis and Proposition \ref{RankOne2ndFFProp} one obtains
\begin{align*}
	(\mathbf{H}_\Sigma^h)^{\tilde{N}} &= (\mathbf{H}_\Sigma^g)^{\tilde{N}} -(2 \beta-|\mathbf{T}^\top|_g^2\beta^2) (\mathbf{A}_\Sigma^g(\mathbf{T}^\top,\mathbf{T}^\top))^{\tilde{N}}-\frac{2(\mathbf{a}(\mathbf{T}^\top))^{\tilde{N}}}{1+ |\mathbf{T}^\top|_g^2} \\
	&=(\mathbf{H}_\Sigma^g)^{\tilde{N}}- \frac{(\mathbf{A}_\Sigma^g(\mathbf{T}^\top, \mathbf{T}^\top))^{\tilde{N}}}{1+|\mathbf{T}^\top|_g^2} -\frac{2}{1+|\mathbf{T}^\top|^2} (\mathbf{a}(\mathbf{T}^\top))^{\tilde{N}}.
\end{align*}
This gives the first equation.

In the same manner, one obtains
\begin{align*}
  g(\mathbf{H}_\Sigma^h, \mathbf{T}^{\hat{N}})&=	 g(\mathbf{H}_\Sigma^h, \mathbf{T})-(2 \beta-|\mathbf{T}^\top|_g^2\beta^2) g(\mathbf{A}_\Sigma^g(\mathbf{T}^\top, \mathbf{T}^\top), \mathbf{T})-\frac{2 g(\mathbf{T}^{\hat{N}}, \mathbf{a}(\mathbf{T}^\top))}{1+|\mathbf{T}|_g^2}\\
  &=g( \mathbf{H}_{\Sigma}^g, \mathbf{T})- \frac{g(\mathbf{A}_\Sigma^g(\mathbf{T}^\top, \mathbf{T}^\top), \mathbf{T})}{1+|\mathbf{T}^\top|_g^2}-\frac{2g(\mathbf{T}, \mathbf{a}(\mathbf{T}^\top))}{1+|\mathbf{T}|_g^2}.
\end{align*} 
Here the second equality used the anti-symmetry of $\mathbf{a}$ to see that
\begin{align*}
g(\mathbf{T}^{\hat{N}}, \mathbf{a}(\mathbf{T}^\top))&= g(\mathbf{T}, \mathbf{a}(\mathbf{T}^\top))-\frac{1+|\mathbf{T}|^2_g}{1+|\mathbf{T}^\top|_g^2} g(\mathbf{T}^\top, \mathbf{a}(\mathbf{T}^\top)) = g(\mathbf{T}, \mathbf{a}(\mathbf{T}^\top)).
\end{align*}
\end{proof}

\bibliographystyle{hamsabbrv}
\bibliography{Library2}
\end{document}